\documentclass[11pt]{article}

\usepackage{amssymb}
\usepackage{amsmath}
\usepackage{amsthm}

\usepackage{stmaryrd}
\usepackage[all]{xy}
\usepackage{rotating}
\usepackage{slashed}

\newcommand{\N}{\mathbb{N}}
\newcommand{\Z}{\mathbb{Z}}
\newcommand{\R}{\mathbb{R}}
\newcommand{\C}{\mathbb{C}}

\newcommand{\cfrak}{\mathfrak{c}}

\newcommand{\cbar}{\overline{\mathfrak{c}}}

\newcommand{\fA}{\mathfrak{A}}
\newcommand{\B}{\mathfrak{B}}
\newcommand{\K}{\mathfrak{K}}

\newcommand{\cS}{\mathcal{S}}

\newcommand{\cA}{\mathcal{A}}
\newcommand{\cH}{\mathcal{H}}

\newcommand{\Cl}{{\C\ell}}

\newcommand{\bigcupdot}{\dot{\bigcup}}
\newcommand{\hatotimes}{\mathbin{\widehat\otimes}}
\newcommand{\hatboxtimes}{\mathbin{\widehat\boxtimes}}

\DeclareMathOperator{\id}{id}
\DeclareMathOperator{\im}{im}
\DeclareMathOperator{\ind}{ind}
\DeclareMathOperator{\End}{End}
\DeclareMathOperator{\supp}{supp}
\DeclareMathOperator{\dist}{dist}
\DeclareMathOperator{\grad}{grad}

\theoremstyle{plain}
\newtheorem{thm}{Theorem}[section]
\newtheorem{lem}[thm]{Lemma}
\newtheorem{prop}[thm]{Proposition}
\newtheorem{cor}[thm]{Corollary}

\newtheorem*{mainthm1}{Theorem \ref{mainthm:indexformula}}
\newtheorem*{mainthm2}{Theorem \ref{thm:lambdaBundleTwist}}
\newtheorem*{mainthm3}{Theorem \ref{thm:Inverselimitindexformula}}
\newtheorem*{mainthm4}{Theorem \ref{mainthm:modKindexformula}}
\newtheorem*{mainthm5}{Theorem \ref{mainthm:moduleindexformula}}
\newtheorem*{mainthm6}{Theorem \ref{thm:pairinginterpretation}}
\newtheorem*{mainthmbordism}{Theorem
\ref{thm:BordismInvariance}}

\theoremstyle{definition}
\newtheorem{defn}[thm]{Definition}
\newtheorem{exam}[thm]{Example}
\newtheorem*{acknow}{Acknowledgements}

\theoremstyle{remark}
\newtheorem{rem}[thm]{Remark}

\title{Coarse indices of twisted operators}
\author{Christopher Wulff\thanks{Supported by the Program of Post-Doctoral
Scholarships at the Universidad Nacional Aut\'onoma de M\'exico.}}

\newtheorem*{msc}{Mathematics Subject Classification (2010)}
\newtheorem*{keyws}{Keywords}

\begin{document}

\maketitle
\begin{abstract}
Several formulas for computing coarse indices of twisted Dirac type operators are introduced. One type of such formulas is by composition product in $E$-theory. 
The other type is by module multiplications in $K$-theory, which also yields an index theoretic interpretation of the duality between Roe algebra and stable Higson corona.
\end{abstract}

\begin{msc}
58J22, 19K56, 19K35.
\end{msc}
\begin{keyws}
Coarse index theory, twisted Dirac operators, index formulas, $E$-theory, stable Higson corona.
\end{keyws}

\section*{Introduction}
The purpose of this paper is to give several $E$- and $K$-theoretic formulas for calculating coarse indices of twisted Dirac operators over complete Riemannian manifolds, which eventually also enable us to give an index theoretic interpretation of the duality between Roe algebra and stable Higson corona (cf.\ \cite{EmeMey}).

Recall that the index of twisted longitudinally elliptic  operators $D_E$ over foliations $(M,\mathcal{F})$ can be calculated from a fundamental class $[D]\in KK(C(M),C^*(M,\mathcal{F}))$ of the original operator $D$ and the $K$-theory class $[E]\in K^0(M)$ of the vector bundle $E\to M$ by composition product in $KK$-theory: 
\begin{equation*}
\ind(D_E)=[D]\circ [E]
\end{equation*}
(cf.\ \cite[Section 8.2]{Kordyukov}, \cite[Theorem 13]{Kucerovsky}).
This formula generalizes the index pairing between $K$-homology and $K$-theory, which allows to calculate the index of a twisted ($\Z_2$-)graded Dirac operator by $\ind(D_E)=\langle [D],[E]\rangle\in\Z$ from the fundamental class $[D]\in K_0(M)$ and again $[E]\in K^0(M)$.

A similar formula in coarse index theory had been missing so far. 
A fundamental class $[D]\in K_0(M)$ of a Dirac operator over a complete Riemannian manifold $M$ does exist and one obtains the coarse index $\ind(D)\in K_0(C^*(M))$ from it as the image under the coarse assembly map $\mu\colon K_*(M)\to K_*(C^*(M))$ \cite[Chapter 12]{HigRoe}. 
But at first glance\footnote{For a second glance, see Section \ref{sec:descent}.}, this fundamental class seems inappropriate for calculating coarse indices of twisted operators for two different reasons: first, 
one might be interested in twisting by more general vector bundles $E\to M$ than those which are trivial outside of a compact subset $K\subseteq M$ and thus yield elements of $K^0(M)$
and, second, the coarse index of the twisted operator is an element of $K_*(C^*M))$ whereas the index pairing yields only an integer.

Instead one might hope for a fundamental class in $KK(C_b(M,\K),C^*(M))$, because arbitrary vector bundles over $M$ correspond to projection valued functions $M\to\K$, $\K$ being the compact operators.
However, the $C^*$-algebra $C_b(M,\K)$ is too large to hope for such an element. First of all, it is non-separable, which is a problem because the definition of the composition product in $KK$-theory relies on separability. We circumvent this problem by using $E$-theory instead. The concrete picture of $E$-theory which is suitable for non-separable $C^*$-algebras will be reviewed in Section \ref{sec:Etheory}.

Second,  $C_b(M,\K)$ is too large in the sense that our construction yields only a fundamental class $\llbracket D\rrbracket\in E(\cA(M),C^*(M))$, where $\cA(M)\subseteq C_b(M,\K)$ is the sub-$C^*$-algebra generated by the smooth functions with bounded gradient (which is still non-separable). 
If $M$ is a manifold of bounded geometry, then $\cA(M)$ consists exactly of the uniformly continuous bounded functions $M\to\K$ and $K(\cA(M))$ is a stabilized version of the uniform $K$-theory $K_u^0(M)$ defined in \cite[Definition 4.1]{AlexUniform}.
A smooth projection $P\in\cA(M)$ yields a smooth vector bundle $E\to M$ with fibres $E_x:=\im(P(x))$ --we call such bundles vector bundles of bounded variation-- and a $K$-theory class $\llbracket E\rrbracket:=\llbracket P\rrbracket\in K(\cA(M))$.
The index of the twisted operator $\ind(D_E)$ is independent of the choice of a connection on $E$ because of bordism invariance and the expected index formula holds: $\ind(D_E)=\llbracket D\rrbracket\circ\llbracket E\rrbracket\in K(C^*(M))$.

In fact, we will use a much more general set-up by considering Dirac bundles whose fibres are finitely generated graded projective Hilbert-$A$-modules, $A$ a unital graded $C^*$-algebra. Indices of the associated $A$-linear Dirac operators $D$ are then elements of the $K$-theory of the Roe algebra with coefficients in $A$, $\ind(D)\in K(C^*(M,A))$. In particular this covers $\Cl_k$-linear Dirac operators, thereby also taking care of the index theory of ungraded Dirac operators by Bott periodicity.
Given another unital graded $C^*$-algebra $B$, we define $\cA^\infty(M,B)$ to be the algebra of functions $f\in C_b^\infty(M,B\hatotimes\widehat\K)$ with bounded gradient and $\cA(M,B)$ its closure in $C_b(M,B\hatotimes\widehat\K)$.
Here, $\widehat\K$ is the \emph{graded} $C^*$-algebra of compact operators.
The fundamental class of Definition \ref{def:ETheoryClassBoundedGradient} is then an element
$\llbracket D;B\rrbracket\in E(\cA(M,B),C^*(M,A\hatotimes B))$.

Although we define these $E$-theory classes by direct construction,
we shall establish a ``descent'' homomorphism 
\[K_n(M)\to E(\cA(M,B),C^*(M,\Cl_n\hatotimes B))\]
for manifolds of bounded geometry
in Section \ref{sec:descent} which allows to calculate the $E$-theory class $\llbracket D,B\rrbracket$ of a $\Cl_n$-linear Dirac operator $D$ from its $K$-homology class $[D]$.
It probably generalizes in some way to general $A$-linear Dirac operators, but we will not be pursue this any further in this paper, because we are happy with the direct definition of the $E$-theory classes.

Our first main theorem is the following index formula. The notion of twisted operator in the context of Hilbert module bundles is introduced in Lemma \ref{lem:twistedConnection} and Definition \ref{def:twistedoperator}.
\begin{mainthm1}
Let $M$ be a complete Riemannian manifold, $S\to M$ a Dirac $A$-bundle with associated Dirac operator $D$
and $E\to M$  a $B$-bundle of bounded variation.
Then the index of the twisted operator $D_E$ is
\[\ind(D_E)=\llbracket D;B\rrbracket\circ\llbracket E\rrbracket\,.\]
\end{mainthm1}

As a main technical tool we will prove a bordism invariance result for the fundamental classes. It lifts the bordism invariance of coarse indices (see also \cite{WulffBordism}) and is also of separate interest.
\begin{mainthmbordism}[Bordism invariance]
Let $W$ be a complete Riemannian manifold with a boundary $\partial W$ which decomposes into two complete Riemanian manifolds $M_0$, $M_1$ and assume that there are collar neighborhoods $M_0\times[0,2\varepsilon)$ of $M_0$ and $M_1\times(-2\varepsilon,0]$ of $M_1$ on which the Riemannian metric of $W$ is the product metric.
Let $S_{0}\to M_0$ and $S_1\to M_1$ be Dirac $A$-bundles with associated Dirac operators $D_0$, $D_1$, respectively, and assume that there is a Dirac $A\hatotimes\Cl_1$-bundle $S\to W$, which restricts to the product Dirac $A\hatotimes\Cl_1$-bundles $S_0\hatboxtimes S_{[0,2\varepsilon)}$ and $S_1\hatboxtimes S_{(-2\varepsilon,0]}$ over the respective collar neighborhoods. 
Then for every graded unital $C^*$-algebra $B$ the following diagram in $E$-theory commutes:
\[\xymatrix@C=3em{
&\cA(M_0,B)\ar[r]_-{\llbracket D_0;B\rrbracket}&C^*(M_0,A\hatotimes B)\ar[dr]_{(M_0\subseteq W)_*}&
\\\cA(W,B)\ar[ur]_{\text{restr.}}\ar[dr]^{\text{restr.}}
&&&C^*(W,A\hatotimes B)
\\&\cA(M_1,B)\ar[r]^-{\llbracket D_1;B\rrbracket}&C^*(M_1,A\hatotimes B)\ar[ur]^{(M_1\subseteq W)_*}&
}\]
\end{mainthmbordism}

The index formula of Theorem \ref{mainthm:indexformula} covers probably the most common bundles $E$, but nevertheless it is possible to derive a similar formula for arbitrary vector bundles. This will be covered in Section \ref{sec:generalbundles}:
Assume that $E$ is given by a smooth projection $P\in C_b^\infty(M,B\hatotimes\widehat\K)$. We can always find a smooth function $\lambda\colon M\to [1,\infty)$ such that 
$\lambda^{-1}\cdot\grad(P)$ is bounded, i.\,e.\ that it is contained in 
\[\cA_\lambda^\infty(M,B):=\{f\in C_b^\infty(M,B\hatotimes\widehat\K)\mid \lambda^{-1}\cdot\grad(f)\text{ is bounded}\}\,.\]
We can treat this case by conformally changing the metric on $M$ by multiplying with $\lambda^2$. Denote $M$ with this new metric by $M^\lambda$ and the suitably changed Dirac operator by $D^\lambda$.
Then the closure $\cA_\lambda(M,B)$ of $\cA^\infty_\lambda(M,B)$ is equal to $\cA(M^\lambda,B)$ and we can apply Theorem \ref{mainthm:indexformula} to obtain:
\begin{mainthm2}
The index of the twisted operator $D_E$ is the $E$-theory product
\[\ind(D_E)=(c_\lambda)_*\circ \llbracket D^\lambda;B\rrbracket \circ \llbracket E\rrbracket\in K(C^*(M,A\hatotimes B))\]
of the element induced by the coarse map $c_\lambda:=\id\colon M^\lambda\to M$, the $E$-theory class $\llbracket D^\lambda;B\rrbracket\in E(\cA(M^\lambda,B),C^*(M^\lambda,A\hatotimes B))$ and the $K$-theory class $\llbracket E\rrbracket\in K(\cA_\lambda(M,B))$ determined by the projection $P$.
\end{mainthm2}
Furthermore, the $(c_\lambda)_*\circ\llbracket D^\lambda;B\rrbracket$ combined yield an element $\underleftarrow{\llbracket D;B\rrbracket}$ in the inverse limit
\[\underleftarrow{E}(C_b(M,B\hatotimes\widehat\K),C^*(M,A\hatotimes B)):=\varprojlim_\lambda E(\cA_\lambda(M,B),C^*(M,A\hatotimes B))\,.\]
This group pairs with 
$K(C_b(M,B\hatotimes\widehat\K))=\varinjlim_\lambda K(\cA_\lambda(M,B))$
to give elements of $K(C^*(M,A\hatotimes B))$, thus condensing Theorem \ref{thm:lambdaBundleTwist} into one formula:
\begin{mainthm3}
If the $B$-bundle $E$ is determined by a smooth projection valued function
$P\colon M\to B\hatotimes\widehat\K$, then 
$\ind(D_E)=\underleftarrow{\llbracket D;B\rrbracket}\circ \llbracket E\rrbracket$, where $\llbracket E\rrbracket$ denotes the $K$-theory class of $P$.
\end{mainthm3}

From Section \ref{sec:dividingoutcompacts} on we change our view of coarse index theory by considering coarse indices not as elements of $K(C^*(M,A))$ but as elements of $K(C^*_{/\K}(M,A))$, where $C^*_{/\K}(M,A)$ is the quotient of $C^*(M,A)$ by the compact operators.
This is motivated by the observation that coarse index theory on compact manifolds behaves quite differently from coarse index theory on non-compact complete Riemannian manifolds.
Indeed, for compact $M$ the coarse index is the same as the classical $A$-index under a canonical isomorphism $K(C^*(M,A))\cong K(A)$ whereas
coarse indices of operators over non-compact connected manifolds only contain large-scale information and everything that happens on compact subsets is irrelevant.
The new index in $K(C^*_{/\K}(M,A))$ does not have this deficiency, because it vanishes whenever $M$ is compact, but preserves all information if $M$ is non-compact and connected.

A big advantage of this new set-up is that Dirac operators $D$ which are defined \emph{outside of a compact subset $K\subseteq M$} now have coarse indices $\ind_{/\K}(D)\in K(C^*_{/\K}(M,A))$. Furthermore, we will construct fundamental $E$-theory classes 
\[\llbracket D;B\rrbracket_{/\K/C_0}\in E(\cA_{/C_0}(M,B),C^*_{/\K}(M,A\hatotimes B))\]
where
$\cA_{/C_0}(M,B):=\cA(M,B)/C_0(M,B\hatotimes\K)$. 
The bundles $E\to M\setminus K$ which determine classes in the $K$-theory of this $C^*$-algebra are the \emph{$B$-bundles of bounded variation defined outside of $K$}.
The corresponding index formula is:
\begin{mainthm4}
Let $M$ be a complete Riemannian manifold, $S\to M\setminus K$ a Dirac $A$-bundle  with Dirac operator $D$ and $E\to M\setminus K$ a $B$-bundle of bounded variation defined outside of a compact subset $K\subseteq M$. Then the index of the twisted operator $D_E$ is
\[\ind_{/\K}(D_E)=\llbracket D;B\rrbracket_{/\K/C_0}\circ \llbracket E\rrbracket_{/C_0}\,.\]
\end{mainthm4}

A completely different type of index formula is obtained in Section \ref{sec:duality} by specializing even further to vector bundles $E\to M\setminus K$ which are defined by functions $P\in\cA(M,B)$ which are projection valued outside of a compact subset $K\subseteq M$ and whose gradients vanish at infinity. 
Assuming that $M$ has bounded geometry, these projections are exactly those mapping to projections in the stable Higson corona with coefficients in $B$ of \cite{EmeMey}: $P\in \cfrak(M,B)$, thus defining $\llbracket E\rrbracket_{\cfrak}\in K(\cfrak(M,B))$.

Now, there are associative products
\[K_*(C_{/\K}^*(M,A))\hatotimes K_*(\cfrak(M,B))\to K_*(C_{/\K}^*(M,A\hatotimes B))\,,\]
which are in fact module multiplications in the case $B=\C$. Note that $K_*(C_{/\K}^*(M,A))$ behaves somewhat like a ``coarse homology theory'' while $K_*(\cfrak(M,B))$ behaves like a ``coarse cohomology theory''. The multiplication can thus be interpreted as a cap product.

This module structure allows the calculation of the index of the twisted operator $D_E$ even without knowing the fundamental class of $D$. Only the index of $D$ and the class of $E$ are required:
\begin{mainthm5}
Let $M$ be a complete Riemannian manifold of bounded geometry, $D$ the Dirac operator of a Dirac $A$-bundle $S\to M\setminus K$ and $E\to M\setminus K$ a
$B$-bundle of vanishing variation defined outside of the compact subset $K\subseteq M$.
Then 
\[\ind_{/\K}(D_E)=\ind_{/\K}(D)\cdot \llbracket E\rrbracket_{\cfrak}\,.\]
\end{mainthm5}
This theorem is another instance of 
``calculating indices by module multiplication'', the first being \cite[Corollary 11.3]{WulffFoliations}. Furthermore, the author believes that this type of index formula might also serve as a prototype for similar results in other types of $C^*$-algebraic index theory.

Recall that the stable Higson corona, which appears here, was originally introduced in \cite{EmeMey} to construct a coarse co-assembly map
\[\mu^*\colon \tilde K_{-*+1}(\cfrak(M,B))\to KX^*(M,B)\]
which is dual to the well-studied coarse assembly map
\[\mu\colon KX_*(M,A)\to K_*(C^*(M,A))\,.\]
At first glance it might cause a bit of a stomach-ache why the domain of the coarse co-assembly map is not the $K$-homology group $K^*(C^*(M,B))$.
Our Theorem \ref{mainthm:moduleindexformula} can be seen as an index theoretic interpretation of the duality between Roe algebra and stable Higson corona, suggesting that the stable Higson corona is indeed a good choice. 

Emerson and Meyer motivated this duality furthermore by constructing a pairing between $K_*(C^*(M,\C))$ and $\tilde K_{-*+1}(\cfrak(M,\C))$ which is compatible with assembly, co-assembly and the pairing between $KX_*$ and $KX^*$ in the obvious way. In our context of a non-compact, complete, connected Riemannian manifold of bounded geometry, this pairing corresponds after slight modification to the composition of the module multiplication
$K_*(C_{/\K}^*(M,\C))\hatotimes K_{-*+1}(\cfrak(M,\C))\to K_1(C_{/\K}^*(M,\C))$ and the boundary map $K_1(C_{/\K}^*(M,\C))\to K_0(\widehat\K)\cong\Z$.
These pairings, which we construct in slight more generality in Section \ref{sec:duality}, are stabilized versions of the pairing of \cite{YuKIndices} and may be used to detect non-vanishing of indices.
Finally, our results immediately yield the following index theoretic interpretation of the pairing:
\begin{mainthm6}
Let $M$ be a complete non-compact Riemannian manifold of bounded geometry, $D$ the Dirac operator of a Dirac $A$-bundle $S\to M\setminus K$ and $E\to M\setminus K$ a
$B$-bundle of vanishing variation defined outside of the compact subset $K\subseteq M$.
Then the pairing between $\ind_{/\K}(D)$ and $\llbracket E\rrbracket_{\cfrak}$ is an obstruction to the extendibility of $D_E$ to all of $M$.
\end{mainthm6}

Finally, in  Section \ref{sec:Application} we use the partitioned manifold theorem to show that (in certain set-ups) the connecting homomorphism $\partial\colon K_1(C^*_{/\K}(M,A))\to K_0(A)$ maps the index $\ind_{/\K}(D)$ of an $A\hatotimes\Cl_1$-linear Dirac operator $D_M$ defined outside of a compact subset to the $A$-index of a suitable $A$-linear Dirac operator $D_N$ on a compact connected codimension one submanifold $N\subset M$.
This map is therefore the natural choice for detecting  non-vanishing of the modulo-$\K$-indices.

Interestingly, this method appears to become very powerful in combination with the module structure: In Example \ref{ex:Hopfbundletwist} we will illustrate this by showing that $\partial$ maps the modulo-$\K$-index of the ($\Cl_3$-linear) spinor Dirac operator $\slashed D$ over $\R^3$ to $0$, but after twisting with the pullback $E\to\R^3\setminus\{0\}$ of the nontrivial vector bundle $H\to S^2$ under the radial projection we have 
$\partial(\ind_{/\K}(\slashed D)\cdot\llbracket E\rrbracket_{\cfrak})=\partial(\ind_{/\K}(\slashed D_E))\not=0$, and this finally implies non-triviality of $\ind_{/\K}(\slashed D)$.

\begin{acknow}
The author is grateful to Thomas Schick, Rudolf Zeidler and Alexander Engel for valuable discussions and helpful comments. 
\end{acknow}

\tableofcontents

\section{The asymptotic category and $E$-theory}\label{sec:Etheory}
The central building block of the results presented in this paper is $E$-theory of non-separable $\Z_2$-graded $C^*$-algebras. For this reason, we begin with a whole section reviewing its definition and basic properties.
This section is in large parts identical with \cite[Section 8]{WulffFoliations}, but the focus will partly be on different properties.

We use the picture of $E$-theory presented in \cite{HigGue}. 
A more detailed exposition of $E$-theory, which is based on a slightly different definition, is found in \cite{GueHigTro}.
Furthermore, we treat $K$-theory of graded $C^*$-algebras simply as the special case $K(B)=E(\C,B)$, which is essentially the spectral picture of $K$-theory of \cite{Trout}.

At this point it is necessary to fix some terminology concerning graded $C^*$-algebras.
In this paper, all $C^*$-algebras are $\Z_2$-graded $C^*$-algebras; ungraded $C^*$-algebras are treated as trivially graded. 
For simplicity we will always drop the ``$\Z_2$-'' and simply speak of ``gradings'' in all contexts.

The tensor product we use is always the \emph{maximal} graded tensor product which we denote by $\hatotimes$.
This is in contrast to the usual convention, which uses this symbol for the minimal tensor product. However, the minimal tensor product does not appear in this paper -- except at one place where it is explicitly mentioned. Recall that there is a canonical symmetry isomorphism $A\hatotimes B\cong B\hatotimes A$ which sends $a\hatotimes b$ to $(-1)^{\partial a\partial b}b\hatotimes a$.

Here are the most important graded $C^*$-algebras, which we encounter in this paper:
\begin{itemize}
\item By $\cS$ we mean the $C^*$-algebra $C_0(\R)$ but equipped with the grading into even and odd functions.
\item Given a graded Hilbert-$A$-module $\cH=\cH^+\oplus\cH^-$, the $C^*$-algebras of compact and adjointable operators are denoted by $\K_A(\cH)$ and $\B_A(\cH)$, respectively, and we consider them equipped with the grading into diagonal and off-diagonal matrices. In the Hilbert space case, i.\,e.\ $A=\C$, we drop the index.
\item Let $\ell^2$ be our favourite ungraded infinite dimensional separable Hilbert space (e.\,g.\ $\ell^2=\ell^2(\N)$) and  $\widehat{\ell^2}=\ell^2\oplus\ell^2$ the graded Hilbert space with even and odd part equal to $\ell^2$. Abbreviate $\K:=\K(\ell^2)$ and $\widehat{\K}:=\K(\widehat{\ell^2})$. Furthermore, for $i,j\in\N$ let $\C^{i,j}:=\C^i\oplus\C^j$ graded as indicated. We obtain the graded matrix algebras $M_{i,j}(\C)=\K(\C^{i,j})=\B(\C^{i,j})$.
\item Given a Euclidean vector space $V$ we denote its complex Clifford algebra by $\Cl(V)$ and equip it with the grading in which $V\subseteq (\Cl(V))^{\mathrm{odd}}$.
 In particular for $V=\R^k$ with standard basis $e_1,\dots,e_k$ we write $\Cl_k:=\Cl(\R^k)$ and recall that it is generated by $e_1,\dots, e_k$ under the relations
$e_ie_j+e_je_i=-2\delta_{ij}$. For  all $k\in\N$ we have $\Cl_{2k}\cong M_{k,k}(\C)$.
Note furthermore that there are canonical $*$-isomorphisms 
\begin{align}
\B_{A\hatotimes\Cl_k}(\cH\hatotimes\Cl_k)
&\cong\B_A(\cH)\hatotimes\Cl_k\qquad
\label{eq:AdjointableClifford}
\end{align}
for every Hilbert $A$-module $\cH$.
\end{itemize}

Finally it should be said that, throughout the whole paper, equivalence classes will always be denoted by overlines.

We can now proceed with the actual topic of this section, namely the asymptotic methods.

\begin{defn}[{\cite[Definition 2.2]{HigGue},\cite[Definition 1.1]{GueHigTro}}]
Let $B$ be a graded $C^*$-algebra. The \emph{asymptotic $C^*$-algebra of $B$} is 
\[\mathfrak{A}(B):=C_b([1,\infty),B)/C_0([1,\infty),B).\]
$\mathfrak{A}$ is a functor from the category of $\Z_2$-graded $C^*$-algebras into itself.

An \emph{asymptotic morphisms} is a graded $*$-homomorphisms $A\to \mathfrak{A}(B)$.
\end{defn}

\begin{defn}[{\cite[Definition 2.3]{HigGue},\cite[Definition 2.2]{GueHigTro}}]
Let $A,B$ be graded $C^*$-algebras.
The \emph{asymptotic functors} $\mathfrak{A}^0,\mathfrak{A}^1,\dots$ are defined by $\mathfrak{A}^0(B)=B$ and
\[\mathfrak{A}^n(B)=\mathfrak{A}(\mathfrak{A}^{n-1}(B)).\]
Two $*$-homomorphisms $\phi^0,\phi^1\colon A\to\mathfrak{A}^n(B)$ are $n$-homotopic if there exists a $*$-homomorphism
$\Phi\colon A\to\mathfrak{A}^n(B[0,1])$, called $n$-homotopy between $\phi^0,\phi^1$, from which the $*$-homomorphisms $\phi^0,\phi^1$ are recovered as the compositions
\[A\xrightarrow{\Phi}\mathfrak{A}^n(B[0,1])\xrightarrow{\text{evaluation at }0,1}\mathfrak{A}^n(B).\]
\end{defn}

\begin{lem}[{\cite[Proposition 2.3]{GueHigTro}}]
The relation of $n$-homotopy is an equivalence relation on the set of $*$-homomorphisms from $A$ to $\mathfrak{A}^n(B)$.
\qed
\end{lem}

\begin{defn}[{\cite[Definition 2.4]{HigGue},\cite[Definition 2.6]{GueHigTro}}]
Let $A,B$  be  graded $C^*$-algebras. Denote by $\llbracket A,B\rrbracket_n$ the set of $n$-homotopy classes of $*$-homomorphisms  from $A$ to $\mathfrak{A}^n(B)$.
\end{defn}

There are two natural transformations $\mathfrak{A}^n\to\mathfrak{A}^{n+1}$:
The first is defined by including $\mathfrak{A}^n(B)$ into $\mathfrak{A}^{n+1}(B)=\mathfrak{A}(\mathfrak{A}^n(B))$ as constant functions.
The second is defined by applying the functor $\mathfrak{A}^n$ to the inclusion of $B$ into $\mathfrak{A}B$ as constant functions.
Both of them define maps 
$\llbracket A,B\rrbracket_n\to\llbracket A,B\rrbracket_{n+1}.$
\begin{lem}[{\cite[Proposition 2.8]{GueHigTro}}]
The above natural transformations define the same map $\llbracket A,B\rrbracket_n\to\llbracket A,B\rrbracket_{n+1}$. \qed
\end{lem}
These maps organize the sets $\llbracket A,B\rrbracket_n$ into a directed system
\[\llbracket A,B\rrbracket_0\to\llbracket A,B\rrbracket_1\to\llbracket A,B\rrbracket_2\to\dots\]
\begin{defn}[{\cite[Definition 2.5]{HigGue},\cite[Definition 2.7]{GueHigTro}}]
Let  $A,B$  be  $\Z_2$-graded $C^*$-algebras. Denote by $\llbracket A,B\rrbracket_\infty$ the direct limit of the above directed system.
We denote the class of a $*$-homomorphism $\phi\colon A\to\mathfrak{A}^n(B)$ by $\llbracket\phi\rrbracket$.
\end{defn}

\begin{prop}[{\cite[Proposition 2.12]{GueHigTro}}]\label{prop:asymptoticcomposition}
Let $\phi\colon A\to\mathfrak{A}^n(B)$ and $\psi\colon B\to\mathfrak{A}^m(C)$ be $*$-homomorphisms. The class $\llbracket\psi\rrbracket\circ\llbracket\phi\rrbracket\in \llbracket A,C\rrbracket_\infty$ of the composite $*$-homomorphism
\[A\xrightarrow{\phi}\mathfrak{A}^n(B)\xrightarrow{\mathfrak{A}^n(\psi)}\mathfrak{A}^{n+m}(C)\]
depends only on the classes $\llbracket\phi\rrbracket\in\llbracket A,B\rrbracket_\infty$, $\llbracket\psi\rrbracket\in\llbracket B,C\rrbracket_\infty$ of $\phi,\psi$. The composition law
\[\llbracket A,B\rrbracket_\infty\times\llbracket B,C\rrbracket_\infty\to \llbracket A,C\rrbracket_\infty,\quad(\llbracket\phi\rrbracket,\llbracket\psi\rrbracket)\mapsto \llbracket\psi\rrbracket\circ\llbracket\phi\rrbracket\]
so defined is associative. \qed
\end{prop}

According to the proposition, we obtain a category:
\begin{defn}[{\cite[Definition 2.6]{HigGue},\cite[Definition 2.13]{GueHigTro}}]
The \emph{asymptotic category} is the category whose objects are graded $C^*$-algebras, whose morphisms are elements of the sets $\llbracket A,B\rrbracket_\infty$, and whose composition law is defined in Proposition \ref{prop:asymptoticcomposition}.
\end{defn}
The identity morphism $1_A\in \llbracket A,A\rrbracket_\infty$ is represented by the identity $\id_A\colon A\to A=\mathfrak{A}^0(A)$.

For arbitrary graded $C^*$-algebras $B,D$, there are canonical asymptotic morphisms
\begin{align*}
\mathfrak{A}(B)\hatotimes  D&\to \mathfrak{A}(B\hatotimes  D)&\bar g\hatotimes  d&\mapsto\overline{t\mapsto g(t)\hatotimes  d}
\\D\hatotimes \mathfrak{A}(B)&\to \mathfrak{A}(D\otimes  B)&d\hatotimes \bar g&\mapsto\overline{t\mapsto d\hatotimes  g(t)}
\end{align*}
and inductively also canonical $*$-homomorphisms $\mathfrak{A}^n(B)\hatotimes  D\to \mathfrak{A}^n(B\hatotimes  D)$, $D\hatotimes \mathfrak{A}^n(B)\to \mathfrak{A}^n(D\hatotimes  B)$.
This is a consequence of {\cite[Lemmas 4.1, 4.2 \& Chapter 3]{GueHigTro}.

\begin{prop}[{\cite[Theorem 4.6]{GueHigTro}}]
The asymptotic category is a mo\-no\-idal category with respect to the \emph{maximal graded tensor product} $\hatotimes $ of $C^*$-algebras and a tensor product on the morphism sets,
\[\hatotimes \colon \llbracket A_1,B_1\rrbracket_\infty\times\llbracket A_2,B_2\rrbracket_\infty\to\llbracket A_1\hatotimes  A_2,B_1\hatotimes  B_2\rrbracket_\infty,\]
with the following property:
If 
$\llbracket\phi\rrbracket\in \llbracket A_1,B_1\rrbracket_\infty$ and $\llbracket\psi\rrbracket\in \llbracket A_2,B_2\rrbracket_\infty$
are represented by $\phi\colon A_1\to\mathfrak{A}^m(B_1)$ and $\psi\colon A_2\to\mathfrak{A}^n(B_2)$ , respectively, and $D$ is another graded $C^*$-algebra, then 
\begin{align*}
\llbracket\phi\rrbracket\hatotimes 1_D&\in \llbracket A_1\hatotimes  D,B_1\hatotimes  D\rrbracket_\infty,
\\1_D\hatotimes  \llbracket\psi\rrbracket&\in \llbracket D\hatotimes  A_2,D\hatotimes  B_2\rrbracket_\infty
\end{align*}
are represented by the compositions
\begin{align*}
A_1\hatotimes  D&\xrightarrow{\phi\hatotimes \id_D}\mathfrak{A}^m(B_1)\hatotimes  D\to\mathfrak{A}^m(B_1\hatotimes  D),
\\D\hatotimes  A_2&\xrightarrow{\id_D\hatotimes \psi}D\hatotimes \mathfrak{A}^n(B_2)\to\mathfrak{A}^n(D\hatotimes  B_2),
\end{align*}
respectively. \qed
\end{prop}
The general form of the tensor product is of course
\[\llbracket\phi\rrbracket\hatotimes \llbracket\psi\rrbracket=(\llbracket\phi\rrbracket\hatotimes 1_{B_2})\circ(1_{A_1}\hatotimes \llbracket\psi\rrbracket)=(1_{B_1}\hatotimes \llbracket\psi\rrbracket)\circ(\llbracket\phi\rrbracket\hatotimes 1_{A_2}).
\]

There is an obvious monoidal functor from the category of graded $C^*$-algebras into the asymptotic category which is the identity on the objects and 
maps a $*$-homomorphism $A\to B$ to its class in $\llbracket A,B\rrbracket_\infty$ by considering it as a $*$-homomorphism $A\to \mathfrak{A}^0(B)$.

The definition of $E$-theory involves the two graded $C^*$-algebras $\widehat{\K}$ and $\cS$ which we saw at the beginning of this section.

The role of $\widehat\K$ is stabilization: Given two separable, graded Hilbert spaces $H_1,H_2$, any isometry $V\colon H_1\hatotimes \widehat{\ell^2}\to H_2\hatotimes \widehat{\ell^2}$ defines an injective $*$-homomorphisms
\[\operatorname{Ad}_V\colon \K(H_1)\hatotimes \widehat\K\to\K(H_2)\hatotimes  \widehat\K,\quad T\mapsto VTV^*.\]
The homotopy class of $\operatorname{Ad}_V$ is independent of the choice of $V$ and therefore defines a canonical isomorphism between $\K(H_1)\hatotimes \widehat\K$ and $\K(H_2)\hatotimes  \widehat\K$ in the asymptotic category. In particular, $\widehat\K\hatotimes \widehat\K$, $\K\hatotimes \widehat\K$, $M_{i,j}(\C)\hatotimes \widehat\K$ and $\Cl_{2k}\hatotimes \widehat\K\cong M_{k,k}(\C)\hatotimes \widehat\K$ are all canonically isomorphic to $\widehat\K$.

Recall from \cite[Section 1.3]{HigGue} that the second $C^*$-algebra, $\cS$, is also a co-algebra with co-unit $\eta\colon \cS\to\C,f\mapsto f(0)$ and a co-multiplication $\Delta\colon \cS\to\cS\hatotimes \cS$,
i.\,e.\ the diagrams
\begin{equation}\label{eq:coalgebraaxioms}
\xymatrix{
\cS\ar[r]^{\Delta}\ar[d]_{\Delta}&\cS\hatotimes \cS\ar[d]^{\id\hatotimes \Delta}&\cS\ar[d]_{\id}\ar[r]^{\id}\ar[dr]^{\Delta}&\cS
\\\cS\hatotimes \cS\ar[r]_{\Delta\hatotimes \id}&\cS\hatotimes \cS\hatotimes \cS&\cS&\cS\hatotimes \cS\ar[l]^{\eta\hatotimes \id}\ar[u]_{\eta\hatotimes \id}
}\end{equation}
commute. As $\cS$ is generated by the two functions
$u(x):=e^{-x^2}$ and $v(x):=xe^{-x^2}$, the $*$-homomorphism $\Delta$ is completely determined by the values
\begin{equation*}
\Delta(u)=u\hatotimes u\,,\quad \Delta(v)=u\hatotimes v+v\hatotimes u\,.
\end{equation*}

\begin{defn}\label{def:Etheory}
Let $A,B$ be graded $C^*$-algebras. The \emph{$E$-theory} of $A,B$  is 
\[E(A,B)=\llbracket\cS\hatotimes  A\hatotimes \widehat\K,B\hatotimes \widehat\K\rrbracket_\infty.\]
It is a group with addition given by direct sum 
 of $*$-homomorphisms 
 \[\cS\hatotimes  A\hatotimes \widehat\K\to \mathfrak{A}^n(B\hatotimes \widehat\K)\]
  (via an inclusion $\widehat\K\oplus\widehat\K\hookrightarrow\widehat\K$, which is canonical up to homotopy) and the zero element represented by the zero $*$-homomorphism.
\end{defn}

\begin{rem}
By \cite[Theorem 2.16]{GueHigTro}, this definition is equivalent to \cite[Definition 2.1]{HigGue} when $A,B$ are separable.
For non-separable $C^*$-al\-ge\-bras, however, it is essential to use Definition \ref{def:Etheory}, because otherwise the products defined below might not exist.
\end{rem}

There is a composition product 
\[E(A,B)\otimes E(B,C)\to E(A,C),\quad (\phi,\psi)\mapsto \psi\circ\phi,\]
where $\psi\circ\phi\in E(A,C)$ is defined to be the composition 
\[\cS\hatotimes  A\hatotimes \widehat\K\xrightarrow{\Delta\hatotimes \id_{A\hatotimes \widehat\K}}\cS\hatotimes \cS\hatotimes  A\hatotimes \widehat\K\xrightarrow{\id_{\cS}\hatotimes \phi}\cS\hatotimes  B\hatotimes \widehat\K\xrightarrow{\psi}C\hatotimes \widehat\K\]
of morphisms in the asymptotic category.

There is also an exterior product
\[E(A_1,B_1)\otimes E(A_2,B_2)\to E(A_1\hatotimes  A_2,B_1\hatotimes  B_2),\quad (\phi,\psi)\mapsto \phi\hatotimes \psi,\]
where $\phi\hatotimes \psi\in  E(A_1\hatotimes  A_2,B_1\hatotimes  B_2)$ is defined to be the composition
\begin{align*}
\cS\hatotimes  A_1\hatotimes  A_2\hatotimes \widehat\K&\xrightarrow{\Delta\hatotimes \id}\cS\hatotimes \cS\hatotimes  A_1\hatotimes  A_2\hatotimes \widehat\K
\cong\cS\hatotimes  A_1\hatotimes \widehat\K\hatotimes \cS\hatotimes  A_2\hatotimes \widehat\K
\\&\xrightarrow{\phi\hatotimes \psi}B_1\hatotimes \widehat\K\hatotimes  B_2\hatotimes \widehat\K
\cong B_1\hatotimes  B_2\hatotimes \widehat\K
\end{align*}
of morphisms in the asymptotic category.

\begin{thm}[{\cite[Theorems 2.3, 2.4]{HigGue}}]
With these composition and exterior products,
the $E$-theory groups $E(A,B)$ are the morphism groups in an additive monoidal category $\mathbf{E}$ whose objects are the graded $C^*$-algebras. \qed
\end{thm}
Here are some properties of $E$-theory.
Our earlier observations imply:
\begin{thm}[Stability]\label{thm:EStability}
For any separable graded Hilbert space $H$, the graded $C^*$-algebra $\K(H)$ is canonically isomorphic in the category $\mathbf{E}$ to $\C$. In particular, this applies to $\widehat\K$, $\K$, $M_{i,j}(\C)$ and $\Cl_{2k}$. \qed
\end{thm}

\begin{thm}[{\cite[Theorems 2.3, 2.4]{HigGue}}]\label{thm:EFunctoriality}
There is a monoidal functor from the asymptotic category into $\mathbf{E}$ which is the identity on the objects and maps $\phi\in\llbracket A,B\rrbracket_\infty$ to the morphism
\[\cS\hatotimes  A\hatotimes \widehat\K\xrightarrow{\llbracket\eta\rrbracket\hatotimes \phi\hatotimes 1_{\widehat\K}}B\hatotimes \widehat\K\]
in the asymptotic category, which we denote by the same letter $\phi$. \qed
\end{thm}

Thus, by taking the $E$-theory product with this $E$-theory element, we obtain homomorphisms 
\[E(D,A)\xrightarrow{\phi\circ}E(D,B),\quad E(B,D)\xrightarrow{\circ\phi}E(A,D)\]
for any third graded $C^*$-algebra $D$.
Consequently, the $E$-theory groups are contravariantly functorial in the first variable and  covariantly functorial in the second variable with respect to morphisms in the asymptotic category and in particular with respect to $*$-homomorphisms.

These functorialities can be computed more easily than arbitrary composition products in $E$-theory:
If $\psi\in E(D,A)$ and $\phi\in \llbracket A,B\rrbracket_\infty$, then $\phi\circ\psi\in E(D,B)$ is the composition
\[\cS\hatotimes  D\hatotimes \K\xrightarrow{\psi}A\hatotimes \widehat\K\xrightarrow{\phi\hatotimes 1_{\widehat\K}}B\hatotimes \K\]
in the asymptotic category. This is, because the co-multiplication $\Delta\colon \cS\to\cS\hatotimes \cS$ in the definition of the composition product in $E$-theory cancels with the co-unit $\eta\colon \cS\to\C$ appearing in the functor from the asymptotic category to $E$-theory by \eqref{eq:coalgebraaxioms}.

Similarly, if $\psi\in E(B,D)$ and $\phi\in \llbracket A,B\rrbracket_\infty$, then $\psi\circ\phi\in E(A,D)$ is the composition
\[\cS\hatotimes  A\hatotimes \widehat\K\xrightarrow{1_{\cS}\hatotimes \phi\hatotimes  1_{\widehat\K}}\cS\hatotimes  B\hatotimes \widehat\K\xrightarrow{\psi}B\hatotimes \widehat\K,\]
and the exterior product of $\phi\in E(A_1,B_1)$ and $\psi\in \llbracket A_2,B_2\rrbracket_\infty$ is the composition
\[\cS\hatotimes  A_1\hatotimes  A_2\hatotimes \widehat\K\xrightarrow{\phi\hatotimes \psi}B_1\hatotimes  B_2\hatotimes \widehat\K.\]

Generalizing the functor from the asymptotic category to the $E$-theory category, 
elements of $E(A,B)$ are also obtained from any morphism in the asymptotic category of the form
\[A\hatotimes \K(H_1)\to B\hatotimes \K(H_2)\quad\text{or}\quad\cS\hatotimes  A\hatotimes \K(H_1)\to B\hatotimes \K(H_2)\]
where $H_1,H_2$ are arbitrary separable, graded Hilbert spaces.
The $E$-theory element is obtained by tensoring with 
$\llbracket\eta\rrbracket\hatotimes \id_{\widehat\K}$ respectively $\id_{\widehat\K}$ and applying stability.

To treat bordisms later on, we need  Bott periodicity. 
Denote by $X$ the odd, unbounded multiplier on $C_0(\R)\hatotimes\Cl_1=C_0(\R,\Cl_1)$ by multiplication with the function $\R\ni x\mapsto x\in(\Cl_1)^{\text{odd}}$ and by 
$D_\R$ the closure of the Dirac operator $e\frac{\partial}{\partial x}\colon C_c(\R,\Cl_1)\to C_c(\R,\Cl_1)$ over $\R$, which is an odd, self-adjoint, regular operator on the Hilbert module $L^2(\R,\Cl_1)$.
\begin{defn}
The \emph{Bott element} $\mathfrak{b}\in E(\C,C_0(\R)\hatotimes\Cl_1)$ is defined by
\[\cS\to C_0(\R)\hatotimes\Cl_1\,,\quad\phi\mapsto \phi(X)\]
and the \emph{dual Bott element} $\mathfrak{b}^*\in E(C_0(\R),\Cl_1)$ is defined by the asymptotic morphism
\[\cS\hatotimes C_0(\R)\to\fA(\K(L^2(\R))\hatotimes\Cl_1)\,,\quad \phi\hatotimes f\mapsto\overline{t\mapsto \phi(t^{-1}D_\R)f}\,.\]
\end{defn}
\begin{thm}[{\cite[Theorem 1.12, Corollary 1.2]{HigGue}}]\label{thm:Bottperiodicity}
The elements $\mathfrak{b}$ and $\mathfrak{b}^*$ are mutually inverse in the sense that
\begin{align*}
(\mathfrak{b}^*\hatotimes 1_{\Cl_1})\circ \mathfrak{b}&\in E(\C,\Cl_2)\cong E(\C,\C)\qquad\qquad\text{and}
\\(\mathfrak{b}\hatotimes 1_{\Cl_1})\circ \mathfrak{b}^*&\in E(C_0(\R),C_0(\R)\hatotimes\Cl_2)\cong E(C_0(\R),C_0(\R))
\end{align*}
are the respective identity morphisms. \qed
\end{thm}

Now that we have introduced all that we need to know about $E$-theory, we simply treat $K$-theory as the special case $K(B):=E(\C,B)$. In this special case, it is not even necessary to use asymptotic morphisms, because
$E(\C,B)\cong \llbracket \cS,B\hatotimes\widehat\K\rrbracket_0$ by \cite[Proposition 1.3]{HigGue}. This is the spectral picture of $K$-theory of \cite{Trout}.

The relation to the classical picture of $K$-theory is the following: If $P\in B\hatotimes\widehat{\K}$ is a projection, then its $K$-theory element is represented by the $*$-homomorphism
\[\cS\to B\hatotimes\widehat{\K}\,,\quad \phi\mapsto \phi(0)\cdot P\,.\]

\section{The Roe algebra with coefficients}
We shall now recall the construction of the Roe algebra  with coefficients.
The coefficient $C^*$-algebras appearing in the remaining part of this paper will usually be denoted by $A,A_1,A_2,B,B_1,B_2,C$. Therefore, we will from now on always assume that these letters stand for graded, unital, complex $C^*$-algebras.

It is important to begin by making some remarks about the spaces $X$ we work with: 
We do not use the notion of a coarse space in this paper, which would be the most general setup. Instead, our spaces will always be proper metric spaces. However, we do allow the metric to take the value infinity, thus yielding a decomposition into subspaces which are infinitely far apart from each other, the so-called coarse components. The reason is that some constructions in the sequel yield non-connected complete Riemannian manifolds, whose path-metric is of this type.
We emphasize right here at the beginning that many arguments in coarse geometry generalize to spaces with such metrics simply by applying them to each coarse component separately. 

The theory which we are about to recall in this section usually makes heavy use of the fact that--if the space has only one coarse component--the bounded subsets are exactly the precompact ones. If the space has more than one coarse component, this has to be adapted to:
\begin{lem}\label{lem:boundednotprecompact}
The precompact subsets are exactly those which are a finite union of bounded subsets. \qed
\end{lem}

That said, complete Riemannian manifolds without boundary are understood to be equipped with the path-metric. Nonetheless, we also encounter spaces with different metrics. For example, if $X$ is the closure of an open subset of a complete Riemannian manifold $M$, then we will usually equip $X$ with the restricted path-metric from $M$. We shall write $X^{\subseteq M}$ instead of $X$ in these cases, to indicate that the restricted metric from $M$ has been chosen.

The basic properties of the Roe algebra stated in this section do not come surprisingly at all, as they are well known in similar set-ups (e.\,g.\ \cite[Chapter 6]{HigRoe}). However, they have not been stated in the context of $E$-theory and in the presence of $\Z_2$-gradings before, and for this reason we will be a bit detailed where it seems appropriate.

\begin{defn}[cf.\ {\cite[Definition 3.2]{HankePapeSchick}, \cite[Definitions 5.1--5.3]{HigsonPedersenRoe}}]
Let $\cH$ be a separable right Hilbert $A$-module,  $\rho\colon C_0(X)\to\B_A(\cH)$ a representation and $T\in\B_A(\cH)$.
\begin{itemize}
\item $T$ is \emph{locally compact} if $T\circ\rho(f),\rho(f)\circ T\in\K_A(\cH)$ for all $f\in C_0(X)$, or equivalently if this holds for all bounded Borel functions $f$ which vanish at infinity.
\item $T$ has \emph{finite propagation} if there exists $R>0$ such that $\rho(f)T\rho(g)$ vanishes for all $f,g\in C_0(X)$ with $\dist(\supp(f),\supp(g))\geq R$. The smallest such $R$ is called the \emph{propagation} of $T$. 
\item The \emph{Roe-$C^*$-algebra} of $X$ associated with $\rho$ is the sub-$C^*$-algebra $C^*(X,\cH)\subseteq \B_A(\cH)$ generated by all locally compact operators with finite propagation. The dependence of the representation $\rho$ is understood implicitly.
\end{itemize}
\end{defn}

There is the well known ``functoriality'' of the Roe algebras under coarse maps.
\begin{defn}
Let $X_1$, $X_2$, $X_3$ be proper metric spaces.
\begin{itemize}
\item A \emph{coarse map} between $X_1,X_2$ is a proper Borel map $h\colon X_1\to X_2$, proper in the sense that preimages of precompact sets are precompact\footnote{In view of Lemma \ref{lem:boundednotprecompact} this is not the same as saying that preimages of bounded sets are bounded.}, such that for each $R>0$ there is $S>0$ such that 
\[\forall x,y\in X_1\colon d(x,y)<R\Rightarrow d(f(x),f(y))<S\,.\]
\item Two coarse maps $h_1,h_2\colon X_1\to X_2$ are called \emph{close} (or \emph{coarsely equivalent}) if the function $X_1\ni x\mapsto d(h_1(x),h_2(x))\in\R$ is bounded. This defines an equivalence relation on the set of coarse maps from $X_1$ to $X_2$.
\item Proper metric spaces together with closeness classes of coarse maps comprise a category. A \emph{coarse equivalence} is a coarse map which is an isomorphism in this category.
\end{itemize}
\end{defn}

\begin{defn}[cf.\ {\cite[Definitions 5.1]{HigsonPedersenRoe}}]
Let $X_1,X_2$ be two proper metric spaces, $A$ a $\Z_2$-graded unital $C^*$-algebra, $\rho_{1,2}\colon C_0(X_{1,2})\to\B_A(\cH_{1,2})$ grading preserving representations on separable $\Z_2$-graded Hilbert-$A$-modules and $V\colon \cH_1\to\cH_2$ an isometry of Hilbert modules.
\begin{itemize}
\item The \emph{support} $\supp(V)$ of $V$ is the complement in $X_1\times X_2$ of the union of all open subsets of the form $U_1\times U_2$ such that $\rho_2(f_2)V\rho_1(f_1)=0$ for all $f_{1,2}\in C_0(U_{1,2})$.
\item A coarse map $h\colon X_1\to X_2$ is \emph{covered} by $V$, if there is an $R>0$ such that 
\[\forall (x,y)\in\supp(V)\colon d(h(x),y)<R\,.\]
\end{itemize}
\end{defn}

\begin{lem}\label{lem:coveringIsometries}
Given an isometry $V$ which covers a coarse map $h$ as in the previous definition, adjoining by $V$ yields a $*$-homomorphism
\[\operatorname{Ad}_V\colon C^*(X_1,\cH_1)\to C^*(X_2,\cH_2)\,,\quad T\mapsto VTV^*\,.\]
The $E$-theory class $h_*\in E(C^*(X_1,\cH_1), C^*(X_2,\cH_2))$ determined by this $*$-homomorphism is independent of the choice of $V$ and depends only on the closeness class of $h$. If $X_1,X_2,X_3$ are three proper metric spaces with representations on Hilbert $A$-modules $\cH_1,\cH_2,\cH_3$ and $h_1\colon X_1\to X_2$ and $h_2\colon X_2\to X_3$ are two coarse maps, which are covered by isometries $\cH_1\to\cH_2$, $\cH_2\to\cH_3$, respectively, then $(h_2\circ h_1)_*=(h_2)_*\circ (h_1)_*$.
\end{lem}
\begin{proof}
The first part was already mentioned in \cite[Section 5]{HigsonPedersenRoe}.
If $V_0,V_1\colon H_1\to H_2$ are two isometries which cover $h$, then $\operatorname{Ad}_{V_0}$, $\operatorname{Ad}_{V_1}$ determine the same $E$-theory elements as the $*$-homomorphisms 
\[\begin{pmatrix}
\operatorname{Ad}_{V_0}&0\\0&0
\end{pmatrix}\,,\,
\begin{pmatrix}
0&0\\0&\operatorname{Ad}_{V_1}
\end{pmatrix}\colon 
C^*(M_1,\rho_1)\to C^*(M_2,\rho_2)\hatotimes M_2(\C)\,,
\]
repectively, but these are homotopic via the $*$-homomorphisms obtained by adjoining with $V_t:=\begin{pmatrix}
\cos(\frac{\pi}{2}t)V_0\\\sin(\frac{\pi}{2}t)V_1
\end{pmatrix}$, $t\in[0,1]$, and thus determine the same element in $E$-theory.

If $V$ covers the coarse map $h_1$ and the coarse map $h_2$ is close to $h_1$, then $V$ obviously covers $h_2$, too.

Finally, if $V_1\colon \cH_1\to\cH_2$, covers $h_1$ and $V_2\colon \cH_2\to\cH_3$ covers $h_2$, then $V_2V_1$ covers $h_2\circ h_1$.
\end{proof}

Isometries which cover coarse maps do not always exist. They do, however, exist if the Hilbert module $\cH_2$ is ``sufficiently large'' \cite[Proposition 5.5]{HigsonPedersenRoe}. 
In the cases mentioned at the beginning, i.\,e.\ a manifold with or without boundary or, more general, closures of open subsets of manifolds, then it is easy to give sufficiently large Hilbert modules explicitely, for example $L^2(X)\hatotimes A\hatotimes\widehat{\ell^2}$. 
We will use this Hilbert module for constructing the Roe algebra with coefficients concretely: 
\begin{defn}\label{def:concreteRoealgebra}
If $X$ is the closure of an open subset of a complete, connected Riemannian manifold and $A$ a graded unital $C^*$-algebra, then the Roe algebra of $X$ with coefficients in $A$ is 
\[C^*(X,A):=C^*(X,L^2(X)\hatotimes A\hatotimes\widehat{\ell^2}))\,.\]
\end{defn}
\begin{lem}\label{lem:coarseFunctoriality}
Given spaces $X_1,X_2$  as in the previous definition, a representation of $C_0(X_1)$ on a separable Hilbert $A$-module $\cH$ and a coarse map $h\colon X_1\to X_2$, then there exists an isometry $\cH\to L^2(M)\hatotimes A\hatotimes\widehat\ell^2$ covering $h$ and hence the $E$-theory element
\[h_*\in E(C^*(X_1,\cH),C^*(X_2,A))\]
exists.
In particular, the assignments $X\to C^*(X,A)$ and $h\mapsto h_*$ define a functor from the category of spaces as in the previous definition and closeness classes of coarse maps to the category of graded $C^*$-algebras.
\end{lem}

\begin{proof}
The existence of an isometry covering $h$ was shown in \cite[Proposition 5.5]{HigsonPedersenRoe} in the case of ungraded coefficient $C^*$-algebras. The gradings are easily incorporated into the proof.
\end{proof}

In the context of index theory, it is usually helpful to choose Hilbert modules which take the presence of bundles $S\to X$ into account. Typically, these bundles are of the following type:
\begin{defn}
Let $X$ be as in Definiton \ref{def:concreteRoealgebra}.
\begin{itemize}
\item By an \emph{$A$-bundle} $S\to X$ we mean a smooth bundle whose fibres are finitely generated, projective, graded right Hilbert $A$-modules.
\item The Hilbert $A$-module $L^2(S)$ or $L^2(X,S)$ is the completion of $\Gamma^\infty_{\operatorname{cpt}}(S)$ with respect to the $A$-valued scalar product
\[(\xi,\zeta)=\int_M\langle\xi(x),\zeta(x)\rangle d\!\operatorname{vol}\,,\]
where $\langle.,.\rangle\colon \Gamma(S)\times\Gamma(S)\to C^\infty(X,A)$ is the fibre-wise $A$-valued inner product.
\item Furthermore, we define $C^*(X,S):=C^*(X,L^2(S)\hatotimes\widehat\ell^2)$.
\end{itemize}
\end{defn}

There are inclusions 
\begin{equation*}
C^*(X,L^2(S))\subseteq C^*(X,S)\subseteq C^*(X,A)
\end{equation*}
given by adjoining with grading preserving isometric embedding of bundles $S\subseteq S\hatotimes\widehat\ell^2\subseteq A\hatotimes\widehat\ell^2$.
Note that the $E$-theory elements corresponding to these inclusions are in fact both of the form $\id_*$ and therefore independent of the choices.
If the fibres of $S$ are ``sufficiently large'' in the appropriate sense, then these $E$-theory elements are even isomorphisms, but we will not use this fact.

As we will see soon, $C^*(X,L^2(S))$ is perfect for the direct construction of the coarse index (cf.\ \cite{HankePapeSchick}), but the extra space in $C^*(X,S)=C^*(X,L^2(S)\hatotimes\widehat\ell^2)$ is very useful when considering twisted operators.

Important coefficient $C^*$-algebras are of course the Clifford algebras $\Cl_k$. Note that the isomorphisms  \eqref{eq:AdjointableClifford} carry over to the Roe algebras:
\begin{lem}\label{lem:RoeClifford}
There are canonical isomorphisms
\begin{align*}
C^*(X,\cH\hatotimes\Cl_k)&\cong C^*(X,\cH)\hatotimes\Cl_k
\end{align*}
which are natural under the functoriality of Lemma \ref{lem:coveringIsometries}. \qed
\end{lem}

\section{Dirac operators over $C^*$-algebras}
The index theorems presented in this paper work for $A$-linear Dirac operators acting on sections of $A$-bundles.
Let's make all these notions precise.
\begin{defn}
Let $S\to M$ be a smooth $A$-bundle. A \emph{connection} on $S$ is a $\C$-linear map $\nabla\colon \Gamma(S)\to\Gamma(S\otimes T^*M)$ which is grading preserving, i.\,e.\ maps $\Gamma(S^\pm)$ to $\Gamma(S^\pm\otimes T^*M)$, satisfies the Leibniz rule $\nabla_X(\xi\cdot f)=\nabla_X(\xi)\cdot f+\xi\cdot \partial_Xf$
for all sections $\xi\in\Gamma(S)$, all smooth $A$-valued functions $f\in C^\infty(M,A)$ and all $X\in TM$, and is metric with respect to the $A$-valued inner product $\langle.,.\rangle$ on the fibres, i.\,e.\ 
$\partial_X\langle\xi,\zeta\rangle=\langle\nabla_X\xi,\zeta\rangle+\langle \xi,\nabla_X\zeta\rangle$ for all $\xi,\zeta\in\Gamma(S)$ and $X\in TM$.
\end{defn}

\begin{defn}
Let $(M,g)$ be a Riemannian manifold (possibly with boundary) and $\nabla^{LC}$ its Levi--Civita connection.
A \emph{Dirac $A$-bundle} is a smooth $A$-bundle along with 
\begin{itemize}
\item a connection $\nabla\colon \Gamma(S)\to\Gamma(S\otimes T^*M)$;
\item a Clifford multiplication 
$\Cl(TM)\to\End_A(S)$, i.\,e.\ a graded $C^*$-algebra bundle homomorphism;
\end{itemize}
such that connection and Clifford multiplication are related by the Leibniz rule 
$\nabla_X(Y\cdot\xi)=(\nabla^{LC}_XY)\cdot\xi+Y\cdot\nabla_X\xi$
for all $X\in TM$, $Y\in\Gamma(\Cl(TM))$ and $\xi\in\Gamma(S)$.
\end{defn}

\begin{defn}
The \emph{Dirac operator} associated to a Dirac $A$-bundle $S\to M$ is the $A$-linear first order differential operator $D\colon \Gamma^\infty_{\operatorname{cpt}}(S)\to\Gamma^\infty_{\operatorname{cpt}}(S)$ which is locally given by
\[D=\sum_{i=1}^ne_i\nabla^S_{e_i}\,,\]
where $e_1,\dots, e_n$ is a local orthonormal frame of $TM$.
\end{defn}

The following two theorems are proven in \cite[Theorem 2.3, Lemma 3.6]{HankePapeSchick} for ($\C$-linear) Dirac operators twisted by $A$-bundles. However, their proofs never make use of the special shape of the twisted operator $D_E$ in any essential way and thus carries over to $A$-linear Dirac operators. Alternatively, see \cite[Lemma 2.1]{Zadeh}, but \cite[Remark 1.6]{HankePapeSchick} pointed out some drawbacks of Zadeh's proof. 
A newer proof of the first of the two theorems in its full generality has recently been provided by Ebert in \cite{Ebert_AnalyticalFoundations}.
\begin{thm}[{cf.\ \cite[Theorem 2.3.]{HankePapeSchick} and \cite[Theorem 1.14]{Ebert_AnalyticalFoundations}}]
The Dirac operator $D$ is closable in $L^2(S)$ and the minimal closure is regular  and self-adjoint as unbounded Hilbert-$A$-module operator.  It is the unique self-adjoint extension of $D$. \qed
\end{thm}
We denote the closure by the same letter $D$. The importance of self-adjoint regular operators on Hilbert modules is that they admit a functional calculus, see  \cite[Theorem 3.1]{HankePapeSchick} and \cite[Theorem 1.19]{Ebert_AnalyticalFoundations}. Relevant to us are the properties mentioned in the following theorem:
\begin{thm}[{cf.\ \cite[Lemma 3.6.]{HankePapeSchick}}]\label{thm:FuncCalc}
The functional calculus yields a graded $*$-homomorphism
\[\cS\to C^*(M,L^2(S))\,,\quad f\mapsto f(D)\]
with the property that $f(D)=Dg(D)=g(D)D$ if $f,g\in\cS$ satisfy $f(t)=tg(t)$ for all $t$.\qed
\end{thm}

\begin{defn}
The coarse index $\ind(D)\in K(C^*(M,A))$ of $D$ is the $K$-theory class obtained by composing the $*$-homomorphism of the previous theorem with the inclusion $C^*(M,L^2(S))\subseteq C^*(M,A)$.
\end{defn}
Sometimes we will also call the class in $K(C^*(M,L^2(S)))$ or $K(C^*(M,S))$ the coarse index. However, we emphasize that we will use functoriality in special cases later on, and this is in general only available for $C^*(M,A)$.

Let us now proceed to twisted operators. The following works for the minimal as well as the maximal tensor product:
\begin{lem}\label{lem:twistedConnection}
Let $S\to M$ be a $A$-bundle with connection $\nabla^S$ and $E\to M$ a $B$-bundle with connection $\nabla^E$. Then there is a unique connection $\nabla^{S\hatotimes E}$ on the $A\hatotimes B$-bundle $S\hatotimes E$ such that 
\begin{equation}\label{eq:twistedConnection}
\nabla^{S\hatotimes E}(\xi\hatotimes\zeta)= \nabla^S\xi\hatotimes\zeta+\xi\hatotimes\nabla^E\zeta
\end{equation}
for all section $\xi\in\Gamma(S)$ and $\zeta\in\Gamma(E)$.

If $S$ is even a Dirac $A$-bundle, then $S\hatotimes E$ is a  Dirac $A\hatotimes B$-bundle with this connection and the obvious Clifford action.
\end{lem}
\begin{proof}
Uniqueness: If there are two connections of this type, then their difference is an endomorphism of the bundle $S\hatotimes E$ which vanishes on the algebraic tensor product of the fibres and therefore on all of $S\hatotimes E$.

To prove existence, we could construct the connection using local trivializations, but we prefer to proceed differently:

As the fibres of the bundles are finitely generated projective Hilbert modules, we can find smooth adjointable isometric and grading preserving inclusions into  trivial bundles
\[V\colon S\xrightarrow{\subseteq} M\times A\hatotimes\widehat\ell^2\,, \quad W\colon E\xrightarrow{\subseteq} M\times B\hatotimes\widehat\ell^2\,.\]
These yield new metric connections on $S$, $E$ and $S\hatotimes E$, namely
\begin{align*}
\tilde\nabla^S_X&:=V^*\circ\partial_X\circ V\,, \qquad\tilde\nabla^E_X:=W^*\circ\partial_X\circ W\,,\\&\tilde\nabla^{S\hatotimes E}_X:=(V\hatotimes W)^*\circ\partial_X\circ (V\hatotimes W)\,,
\end{align*}
respectively. The differences  $\sigma(X):=\nabla^S_X-\tilde\nabla^S_X$ and $\eta(X):=\nabla^E_X-\tilde\nabla^E_X$ are bundle endomorphism of $S$, $E$, respectively. They are antiselfadjoint, because the connections are metric. Therefore, $\sigma(X)\hatotimes\id_E+\id_S\hatotimes \eta(x)$ is an antiselfadjoint bundle homomorphism of $S\hatotimes E$.
One now easily verifies that the metric connection
\[\nabla^{S\hatotimes E}=\tilde\nabla^{S\hatotimes E}+\sigma\hatotimes\id_E+\id_S\hatotimes \eta\]
satisfies Equation \eqref{eq:twistedConnection}.

If $S$ is Dirac, then the induced Clifford action on $S\hatotimes E$ is compatibe with the new connection, too, as is easily verified using Equation \eqref{eq:twistedConnection}.
\end{proof}

\begin{defn}\label{def:twistedoperator}
Let $S\to M$ be a Dirac $A$-bundle, $D$ its associated Dirac operator and $E\to M$ a $B$-bundle with connection. Then the \emph{twisted operator} $D_E$ is the Dirac operator associated to the Dirac $A\hatotimes B$-bundle $S\hatotimes E$.
\end{defn}

Similar to twisted Dirac bundles we can also construct external tensor products of Dirac bundles. The proof of the following lemma is completely analogous to the proof of Lemma \ref{lem:twistedConnection}.
\begin{lem}
Let $S_i\to M_i$ be $A_i$-bundles with connections $\nabla^{S_i}$ for $i=1,2$. Then there is a unique connection $\nabla^{S_1\hatboxtimes S_2}$ on the $A_1\hatotimes A_2$-bundle $S_1\hatboxtimes S_2\to M_1\times M_2$ such that 
\begin{equation*}
\nabla^{S_1\hatboxtimes S_2}_{X_1+X_2}(\xi_1\hatboxtimes\xi_2)= \nabla^{S_1}_{X_1}\xi_1\hatboxtimes\xi_2+ \xi_1\hatboxtimes\nabla^{S_2}_{X_2}\xi_2
\end{equation*}
for all section $\xi_i\in\Gamma(S_i)$ and $X_i\in TM_i$.

If $S_{1,2}$ are even Dirac $A_{1,2}$-bundles, then $S_1\hatboxtimes S_2$ is a  Dirac $A_1\hatotimes A_2$-bundle with this connection and the obvious Clifford action. \qed
\end{lem}

\begin{defn}
Let $S_i\to M_i$ be $A_i$-bundles with associated Dirac operators $D_i$ for $i=1,2$.
Then the \emph{external tensor product} $D_1\times D_2$ is the Dirac operator of the Dirac $A_1\hatotimes A_2$-bundle $S_1\hatboxtimes S_2\to M_1\times M_2$.
\end{defn}
The closure of this operator is a regular operator on the Hilbert $A_1\hatotimes A_2$-module $L^2(M_1\times M_2,S_1\hatboxtimes S_2)\cong L^2(M_1,S_1)\hatotimes L^2(M_2,S_2)$. The following Lemma is more or less well known (cf.\ \cite[Proof of Theorem 4.1]{Dumitrascu}, \cite[Appendix A]{HigsonKasparovTrout}, \cite[Lemma 4.16]{ZeidlerPSCProductSecondary}) and is the key analytic ingredient to bordism invariance.
\begin{lem}\label{lem:externalproductfunctionalcalculus}
The functional calculus of the operators $D_1,D_2$ and $D_1\times D_2$ are related by the formula
\[
\exp(-t^{-2}(D_1\times D_2)^2)=\exp(-t^{-2}D_1^2)\hatotimes\exp(-t^{-2}D_2^2)\,.
\tag*{\qedsymbol}
\]
\end{lem} 

\section{The $E$-theory class of an operator}
With all preparatory work done in the previous three sections, we can now turn our attention to the $C^*$-algebra $\cA(M,B)$ and the $E$-theory classes $\llbracket D;B\rrbracket$ of $A$-linear Dirac operators $D$.

\begin{defn}
Given a Riemannian manifold $M$, possibly with boundary, an \emph{arbitrary} graded $C^*$-algebra $Y$ and a smooth function $f\colon M\to Y$, we define the \emph{gradient} of $f$ to be the section $\grad(f)\in\Gamma(\Cl(TM)\hatotimes Y)$ which is locally defined by the formula 
$\grad(f):=\sum_{k=1}^{\dim M}e_k\hatotimes \partial_{e_k}f$
where $e_1,\dots,e_{\dim M}$ denotes a local orthonormal frame.
\end{defn}
The presence of the grading on $Y$ forces us to be a bit careful with this notion. For example, the following well known result holds in this context for the \emph{graded} commutator. Its proof is a direct calculation.
\begin{lem}\label{lem:commutatorgradient}
If $f\colon M\to Y$ is smooth and $D=\sum_{k=1}^{\dim M}e_k\nabla_{e_k}$ a Dirac operator acting on the smooth sections of a Dirac $A$-bundle $S\to M$, then $f$, $D$ and $\grad(f)$ act in the obvious way on smooth sections of $S\hatotimes Y$ and the action of $\grad(f)$ is equal to the \emph{graded} commutator of $D$ and $f$:
$[D,f]=\grad(f)$.\qed
\end{lem}

\begin{defn}
For a Riemannian manifold $M$, possibly non-complete and with boundary,
and a graded unital $C^*$-algebra $B$ we define
$\cA^\infty(M,B)$ to be the algebra of bounded smooth functions $f\in C_b^\infty(M,B\hatotimes\widehat\K)$ such that
\begin{itemize}
\item $\grad(f)$ is bounded and
\item the continuous extension of $f$ to the completion $\overline{M}$ of $M$, which exists because of the boundedness of the gradient, vanishes on $\overline{M}\setminus M$.
\end{itemize}
By $\cA(M,B)$ we denote the closure of $\cA^\infty(M,B)$ in $C_b(M,B\hatotimes\widehat\K)$.
\end{defn}
Note that the strange second condition is only relevant for non-complete manifolds. Its purpose is to ensure that elements of $\cA(M,B)$ behave like $C_0$-functions on finite scales. 
In particular, if $M$ is embedded Riemann-isometric\footnote{By Riemann-isometric we mean that the Riemannian metric is preserved by the inclusion, but not necessarily the path-metric.} as a codimension zero submanifold into $\tilde M$, then $\cA^\infty(M,B)$ contains $\{f\in\cA^\infty(\tilde M,B)\mid f|_{\tilde M\setminus M}=0\}$ as a norm-dense subalgebra and hence
\[\cA(M,B)=\{f\in\cA(\tilde M,B)\mid f|_{\tilde M\setminus M}=0\}\,.\]

Of particular importance to us are the bundles which have classes in the $K$-theory of these $C^*$-algebras:
\begin{defn}\label{def:bundleboundedvariation}
Let $M$ be complete, $B$ a graded unital $C^*$-algebra and $P\in \cA^{\infty}(M,B)$ a projection. The bundle $E\to M$ consisting of the finitely generated projective Hilbert $B$-modules $E_x:=\im(P(x))\subseteq B\hatotimes\widehat\ell^2$ as fibres
will be called a \emph{$B$-bundle of bounded variation}. Denote by $\llbracket E\rrbracket$ the class of the projection $P$ in $K(\cA(M,B))$.
\end{defn}

The following lemma provides the asymptotic morphism behind the $E$-theory class of an operator: 

\begin{lem}\label{lem:Eclassasymptoticmorphism}
Let $S\to M$ be a Dirac-$A$-bundle over a complete Riemannian manifold $M$ without boundary and $D$ its associated $A$-linear Dirac operator. Then for every unital graded $C^*$-algebra $B$ there is an asymptotic morphism
\begin{align*}
\cS\hatotimes \cA(M,B)&\to\fA( C^*(M,S\hatotimes B)))
\\\phi\hatotimes f&\mapsto\overline{t\mapsto (\phi(t^{-1}D)\hatotimes \id_{B\hatotimes\widehat\K}) \cdot (\id_S\hatotimes f)}\,.
\end{align*}
\end{lem}
From now on we shall simplify the notation and abbreviate 
the canonical actions of $\phi(t^{-1}D)\hatotimes \id_{B\hatotimes\widehat\K}$ and $\id_S\hatotimes f$ on the Hilbert $A\hatotimes B$-module $L^2(M,S)\hatotimes B\hatotimes\widehat\ell^2\cong L^2(M,S\hatotimes B\hatotimes \widehat\ell^2)$ simply by  $\phi(t^{-1}D)$ and $f$, respectively.
\begin{proof}
Consider the two $\Z_2$-graded asymptotic morphisms 
\begin{align*}
\cS&\to\fA(\B_{A\hatotimes B}(L^2(M,S)\hatotimes B\hatotimes\widehat\ell^2))\,,&\phi&\mapsto \overline{t\mapsto\phi(t^{-1}D)}\,,
\\\cA(M,B)&\to\fA(\B_{A\hatotimes B}(L^2(M,S)\hatotimes B\hatotimes\widehat\ell^2))\,,&\phi&\mapsto \overline{t\mapsto f}\,.
\end{align*}
Continuity of $t\mapsto\phi(t^{-1}D)$ follows from norm-continuity of $t\mapsto \phi(t^{-1}\cdot\_)\in\cS$ and the fact that the functional calculus of $D$ is a $*$-homomorphism and therefore continuous.
We have to prove that their images gradedly commute. Let $\phi_{\pm}(x)=(x\pm i)^{-1}$ and $f\in\cA^\infty(M,B)$. If $f$ is of even degree, then
\begin{align*}
[\phi_\pm(t^{-1}D),f]&=(t^{-1}D\pm i)^{-1}f-f(t^{-1}D\pm i)^{-1}
\\&=(t^{-1}D\pm i)^{-1}(f(t^{-1}D\pm i)-(t^{-1}D\pm i)f)(t^{-1}D\pm i)^{-1}
\\&=t^{-1}(t^{-1}D\pm i)^{-1}(-\grad(f))(t^{-1}D\pm i)^{-1}\,.
\end{align*}
If, however, $f$ is of odd degree, then 
\begin{align*}
[\phi_\pm(t^{-1}D),f]&=(t^{-1}D\pm i)^{-1}f-f(-t^{-1}D\pm i)^{-1}
\\&=(t^{-1}D\pm i)^{-1}(f(t^{-1}D\mp i)+(t^{-1}D\pm i)f)(t^{-1}D\mp i)^{-1}
\\&=t^{-1}(t^{-1}D\pm i)^{-1}(+\grad(f))(t^{-1}D\mp i)^{-1}\,.
\end{align*}
Both expressions converge to $0$ in norm as $t\to\infty$. As $\phi_\pm$ generate $\cS$, the universal property of the maximal tensor product is satisfied and we obtain an asymptotic morphism
\[\cS\hatotimes\cA(M,B)\to\fA(\B_{A\hatotimes B}(L^2(M,S)\hatotimes B\hatotimes\widehat\ell^2))\,.\]

By Theorem \ref{thm:FuncCalc}, it is in fact an asymptotic homomorphism to the Roe algebra: If $\phi\in\cS$ and $f\in\cA(M,B)$, then $\phi(t^{-1}D)\in C^*(M,L^2(S))$, i.\,e.\ $\phi(t^{-1}D)$ is the norm limit of locally compact operators $\Phi_n\in \B_A(L^2(M,S))$ with finite propagation $R_n<\infty$, $n\in\N$. Clearly $\phi(t^{-1}D)f$ is the norm limit of the operators $\Phi_nf\in \B_{A\hatotimes B}(L^2(M,S)\hatotimes S\hatotimes \widehat\ell^2)$, which have the same propagations $R_n$. We have to show that they are also locally compact.

For $g\in C_0(M)$ we have $fg\in C_0(M)\hatotimes B\hatotimes\widehat\K$, and therefore $fg=\sum_{m\in\N} f_m\hatotimes T_m$ norm convergent with $f_m\in C_0(M)$ and $T_m\in B\hatotimes \widehat\K$. Hence, $\Phi_nf_m\in\K_A(L^2(M,S))$ and we obtain 
\[(\Phi_nf)g=\sum_{m\in\N}\Phi_nf_m\hatotimes T_m\in \K_A(L^2(M,S))\hatotimes B\hatotimes\widehat\K\,.\]
In order to check that $g(\Phi_n f)$ is compact, too, we may assume without loss of generality that $g$ is compactly supported. Let $g'\in C_0(M)$ be a function which is $1$ on the $R_n$-neighborhood of the support of $g$. Then $g(\Phi_nf)=g(\Phi_nf)g'$ is compact as we have seen above.
\end{proof}

\begin{defn}\label{def:ETheoryClassBoundedGradient}
The asymptotic morphism from the previous lemma determines an $E$-theory class
\[\llbracket D;B\rrbracket\in E(\cA(M,B), C^*(M,S\hatotimes B))\,.\]
We denote the image of this class in $E(\cA(M,B), C^*(M,A\hatotimes B))$ by the same symbol and call it the \emph{$E$-theory class of $D$ with respect to $B$}.
\end{defn}
Note that $K(\cA(M,\C))$ contains a canonical element $1_M$: It is the class of a constant function with value a rank-one projection in $\K=\left(\begin{smallmatrix}\K&0\\0&0\end{smallmatrix} \right)\subseteq\widehat{\K}$. It follows directly from the definition, that $\ind(D)$ is the product of $\llbracket D;\C\rrbracket$ and $1_M$ (and the inclusion $C^*(M,S)\subseteq C^*(M,A)$).

For the sake of dealing with this $E$-theory class, it is very helpful to also define $E$-theory classes for bordisms.

\begin{lem}\label{lem:relativeasymptoticmorphism}
Let $M$ be a complete Riemannian manifold without boundary, $S\to M$ a Dirac $A$-bundle over $M$ with Dirac operator $D$, $U\subseteq M$ an open subset, $B$ a $\Z_2$-graded $C^*$-algebra and $V\colon L^2(\overline{U},S|_{\overline{U}}\hatotimes B)\hatotimes\widehat\ell^2\to L^2(M,S\hatotimes B)\hatotimes\widehat\ell^2$ the isometric inclusion of Hilbert $A\hatotimes B$-modules. Then there is an asymptotic morphism
\begin{align*}
\cS\hatotimes \cA(U,B)&\to\fA( C^*(\overline{U}^{\subseteq M},S|_{\overline{U}}\hatotimes B))
\\\phi\hatotimes f&\mapsto\overline{t\mapsto V^*\phi(t^{-1}D)fV}\,.
\end{align*}
\end{lem}
\begin{proof}
The map 
\[\B_{A\hatotimes B}(L^2(M,S\hatotimes B)\hatotimes\widehat\ell^2)\to\B_{A\hatotimes B}(L^2(\overline{U},S|_{\overline{U}}\hatotimes B)\hatotimes\widehat\ell^2)\,,\quad T\mapsto V^*TV\]
is clearly norm-decreasing and maps $C^*(M,S\hatotimes B)$ to $C^*(\overline{U}^{\subseteq M},S|_{\overline{U}}\hatotimes B))$. Thus, it induces a map
\[\fA(C^*(M,S\hatotimes B))\to\fA( C^*(\overline{U}^{\subseteq M},S|_{\overline{U}}\hatotimes B))\,.\]
We now define the asymptotic morphism by combining the asymptotic isomorphism of Lemma \ref{lem:Eclassasymptoticmorphism} with this map and the inclusion $\cA(U,B)\subseteq\cA(M,B)$.

It is clearly linear and preserves the involution. What remains to be proven is the multiplicativity.  This, however, follows directly from the fact that $fVV^*=f$ for all $f\in\cA(U,B)$.
\end{proof}

\begin{defn}
The asymptotic morphism of the preceding Lemma defines the \emph{$E$-theory class of $D|_U$}
\[\llbracket D|_{U};B\rrbracket\in E(\cA(U,B),C^*(\overline{U}^{\subseteq M},S|_{\overline{U}}\hatotimes B))\,.\]
\end{defn}

The notation $\llbracket D|_U;B\rrbracket$ already indicates that this class depends only on the restriction $D|_U$ of $D$ to the open subset $U$.
This is not completely true, because, first of all, an extension to some $M$ has to exist and, second, $M$ determines the metric on $\overline{U}$. But apart from this, we have the following property:
\begin{lem}[Independence of extension] \label{lem:independence}
Let $U$ be a Riemannian manifold, $S\to U$ a Dirac $A$-bundle with associated Dirac operator $D$ and $B$ a graded unital $C^*$-algebra. Assume that there are Riemann-isometric embeddings $U\subseteq M_1$, $U\subseteq M_2$ into complete Riemannian manifolds $M_1,M_2$ and assume that $S,D$ extend to $S_{1,2},D_{1,2}$ over $M_{1,2}$.
If the path-metric of $M_1$ restricted to $U$ is greater or equal to the path metric of $M_2$ restricted to $U$, then there is a commutative diagram in $E$-theory:
\[\xymatrix@C=2cm{
\cA(U,B)\ar[dr]_-{\llbracket D;B\rrbracket}\ar[r]^-{\llbracket D;B\rrbracket}
&C^*(\overline{U}^{\subseteq M_1},S\hatotimes B)\ar[d]^{\substack{\text{inclusion of $C^*$-algebras}\\=(\id\colon \overline{U}^{\subseteq M_1}\to \overline{U}^{\subseteq M_2})_*}}
\\&C^*(\overline{U}^{\subseteq M_2},S\hatotimes B)
}\]
\end{lem}
Of course, the two sets $\overline{U}^{\subseteq M_{1,2}}$ are the closures in the two distinct manifolds $M_{1,2}$ and therefore it is possible that they are not only different as metric spaces, but also even  as sets.
Note that in the diagram we have simply written $C^*(\overline{U}^{\subseteq M_{1,2}}, S\hatotimes B)$ instead of $C^*(\overline{U}^{\subseteq M_{1,2}}, S_{1,2}|_{\overline{U}^{\subseteq M_{1,2}}}\hatotimes B)$, which is of course unproblematic, because $\overline{U}^{\subseteq M_{1,2}}\setminus U$ is a null set.

\begin{proof}
Denote by $V_{1,2}\colon L^2(U,S\hatotimes B)\hatotimes\widehat\ell^2\to L^2(M_{1,2},S_{1,2}\hatotimes B)\hatotimes\widehat\ell^2$ the isometric inclusions. Let $\phi_{\pm}(x)=(x\pm i)^{-1}$ and $f\in \cA^\infty(U,B)\subseteq \cA(M_{1,2},B)$.

In the following two calculations, as well as later on in this paper, we use the symbol ``$\varpropto$'' to express that the difference between the term to the left and the term to the right vanish in norm in the limit $t\to\infty$. This is, of course, an equivalence relation.

If $f$ is of even degree, we calculate
\begin{align*}
V_2^*\phi_\pm(t^{-1}D_2)fV_2&= V_2^*(t^{-1}D_2\pm i)^{-1}fV_2V_1^*(t^{-1}D_1\pm i)(t^{-1}D_1\pm i)^{-1}V_1
\\&=V_2^*(t^{-1}D_2\pm i)^{-1}V_2\underbrace{V_1^*f(t^{-1}D_1\pm i)}_{=f(t^{-1}D\pm i)V_1^*}(t^{-1}D_1\pm i)^{-1}V_1
\\&=V_2^*(t^{-1}D_2\pm i)^{-1}V_2f(t^{-1}D\pm i)V_1^*(t^{-1}D_1\pm i)^{-1}V_1
\\&\varpropto V_2^*(t^{-1}D_2\pm i)^{-1}V_2(t^{-1}D\pm i)fV_1^*(t^{-1}D_1\pm i)^{-1}V_1
\\&=V_1^*f(t^{-1}D_1\pm i)^{-1}V_1 
\\&\varpropto V_1^*(t^{-1}D_1\pm i)^{-1}fV_1\,.
\end{align*}
If $f$ is of odd degree, the calculation is very similar:
\begin{align*}
V_2^*\phi_\pm(t^{-1}D_2)fV_2&= V_2^*(t^{-1}D_2\pm i)^{-1}fV_2V_1^*(t^{-1}D_1\mp i)(t^{-1}D_1\mp i)^{-1}V_1
\\&=V_2^*(t^{-1}D_2\pm i)^{-1}V_2\underbrace{V_1^*f(t^{-1}D_1\mp i)}_{=f(t^{-1}D\mp i)V_1^*}(t^{-1}D_1\mp i)^{-1}V_1
\\&=V_2^*(t^{-1}D_2\pm i)^{-1}V_2f(t^{-1}D\mp i)V_1^*(t^{-1}D_1\mp i)^{-1}V_1
\\&\varpropto -V_2^*(t^{-1}D_2\pm i)^{-1}V_2(t^{-1}D\pm i)fV_1^*(t^{-1}D_1\mp i)^{-1}V_1
\\&=-V_1^*f(t^{-1}D_1\mp i)^{-1}V_1 
\\&\varpropto V_1^*(t^{-1}D_1\pm i)^{-1}fV_1\,.
\end{align*}
These calculations prove the equality of the two asymptotic morphisms $\cS\hatotimes\cA(U,B)\to\fA(C^*(\overline{U}^{\subseteq M_2},S\hatotimes B))$
behind the diagram in the statement of the lemma.
\end{proof}

\begin{lem}[Restriction]\label{lem:restriction}
Let $M$ be a complete Riemannian manifold without boundary, $S\to M$ a Dirac $A$-bundle with associated Dirac operator $D$, $B$ another graded unital $C^*$-algebra and $U_1\subseteq U_2\subseteq M$ open subsets. Then 
\[\xymatrix@C=2cm{
\cA(U_1,B)\ar[r]^-{\llbracket D|_{U_1};B\rrbracket}\ar[d]^{\text{incl.}}
&C^*(\overline{U_1}^{\subseteq M},S|_{U_1}\hatotimes B)\ar[d]^{\text{incl.}}
\\\cA(U_2,B)\ar[r]^-{\llbracket D|_{U_2};B\rrbracket}
&C^*(\overline{U_2}^{\subseteq M},S|_{U_2}\hatotimes B)
}\]
commutes in $E$-theory.
\end{lem}

\begin{proof}
Let $V_1,V_2$ denote the isometric inclusions
\[L^2(\overline{U}_1,S|_{\overline{U}_1}\hatotimes B)\hatotimes\widehat\ell^2 \xrightarrow{V_1} L^2(\overline{U}_2,S|_{\overline{U}_2}\hatotimes B)\hatotimes\widehat\ell^2 \xrightarrow{V_2} L^2(M,S\hatotimes B)\hatotimes\widehat\ell^2\,.\]
The upper right composition in the diagram is represented by the asymptotic morphism
\begin{align*}
\cS\hatotimes \cA(U_1,B)&\to\fA(C^*(\overline{U_2}^{\subseteq M},S|_{\overline{U}_2}\hatotimes B))
\\\phi\hatotimes f&\mapsto \overline{t\mapsto V_1(V_2V_1)^*\phi(t^{-1}D)f(V_2V_1)V_1^*}\,.
\end{align*}
For every $f\in\cA(U_1,B)$ we have $fV_1V_1^*=f=V_1V_1^*f$
and this implies
\begin{align*}
V_1(V_2V_1)^*\phi(t^{-1}D)f(V_2V_1)V_1^*
&=V_1V_1^*V_2^*\phi(t^{-1}D)V_2fV_1V_1^*
\\&=V_1V_1^*V_2^*\phi(t^{-1}D)V_2f
\\&\varpropto (-1)^{\partial \phi\cdot\partial f}V_1V_1^*fV_2^*\phi(t^{-1}D)V_2
\\&=(-1)^{\partial \phi\cdot\partial f}fV_2^*\phi(t^{-1}D)V_2
\\&\varpropto V_2^*\phi(t^{-1}D)fV_2\,.
\end{align*}
The claim follows, because the lower left composition is represented by the asymptotic morphism
\[\phi\hatotimes f\mapsto \overline{t\mapsto V_2^*\phi(t^{-1}D)fV_2}\,.\qedhere\]
\end{proof}

\begin{lem}[Exterior product]\label{lem:extProd}
Let $M_1,M_2$ be complete Riemannian manifolds without boundary, $S_{1}\to M_{1}$ a Dirac $A_1$-bundle, $S_2\to M_2$ a Dirac $A_2$-bundle, $D_1,D_2$ the associated Dirac operators and $B_1,B_2$ $\Z_2$-graded unital $C^*$-algebras. Consider the exterior product bundle $S_1\hatboxtimes S_2\to M_1\times M_2$ which is a Dirac $A_1\hatotimes A_2$-bundle with associated Dirac operator $D_1\times D_2=D_1\hatotimes 1+1\hatotimes D_2$. Then the following diagram in $E$-theory commutes:
\[\xymatrix@C=3.5cm{
\cA(U_1,B_1)\hatotimes\cA(U_2,B_2)\ar[r]^-{\llbracket D_1|_{U_1};B_1\rrbracket\hatotimes \llbracket D_2|_{U_2};B_2\rrbracket}\ar[d]
&{\begin{array}{l}C^*(\overline{U_1}^{\subseteq M_1},S_1|_{\overline{U_1}}\hatotimes B_1)\hatotimes\\\hatotimes C^*(\overline{U_2}^{\subseteq M_2},S_2|_{\overline{U_2}}\hatotimes B_2)\end{array}}\ar[d]
\\\cA(U_1\times U_2,B_1\hatotimes B_2)\ar[r]^-{\llbracket D_1\times D_2|_{\overline{U_1\times U_2}};B_1\hatotimes B_2\rrbracket}
&{\begin{array}{l}C^*(\overline{U_1\times U_2}^{\subseteq M_1\times M_2},
\\\phantom{C^*(}
S_1\hatboxtimes S_2|_{\overline{U_1\times U_2}}\hatotimes B_1\hatotimes B_2)\end{array}}
}\]
\end{lem}

\begin{proof}
Let $V_{1,2}\colon L^2(\overline{U}_{1,2},S|_{\overline{U}_{1,2}}\hatotimes B_{1,2})\hatotimes\widehat\ell^2\to L^2(M_{1,2},S\hatotimes B_{1,2})\hatotimes\widehat\ell^2$ be the canonical inclusions.
The composition of the upper and the right morphism is represented by composing $\Delta\colon\cS\to\cS\hatotimes\cS$
with the asymptotic morphism
\begin{align*}
\cS\hatotimes\cS&\hatotimes\cA(U_1,B_1)\hatotimes\cA(U_2,B_2)\xrightarrow{\cong} \cS\hatotimes\cA(U_1,B_1)\hatotimes\cS\hatotimes\cA(U_2,B_2)
\\&\to\fA(C^*(\overline{U_1}^{\subseteq M_1},S_1|_{\overline{U_1}}\hatotimes B_1)\hatotimes C^*(\overline{U_2}^{\subseteq M_2},S_2|_{\overline{U_2}}\hatotimes B_2))
\\&\to \fA(C^*(\overline{U_1\times U_2}^{\subseteq M_1\times M_2},
S_1\hatboxtimes S_2|_{\overline{U_1\times U_2}}\hatotimes B_1\hatotimes B_2))
\end{align*}
which maps
\begin{align*}
\phi_1\hatotimes\phi_2&\hatotimes f_1\hatotimes f_2
\mapsto (-1)^{\partial \phi_2\cdot\partial f_1} \phi_1\hatotimes f_1\hatotimes\phi_2\hatotimes f_2
\\&\mapsto (-1)^{\partial \phi_2\cdot\partial f_1}\overline{t\mapsto V_1^*\phi_1(t^{-1}D_1)f_1V_1 \hatotimes V_2^*\phi_2(t^{-1}D_2)f_2V_2}
\\&\mapsto \overline{t\mapsto (V_1\hatotimes V_2)^*(\phi_1(t^{-1}D_1)\hatotimes \phi_2(t^{-1}D_2))(f_1\hatotimes f_2)(V_1\hatotimes V_2)}\,.
\end{align*}
The composition of the left and the lower morphism is simply
\begin{align*}
\cS\hatotimes\cA(U_1,B_1)&\hatotimes\cA(U_2,B_2)
\to\fA(C^*(\overline{U_1\times U_2}^{\subseteq M_1\times M_2},
\\&\qquad\qquad\qquad\phantom{\to\fA(C^*(}S_1\hatboxtimes S_2|_{\overline{U_1\times U_2}}\hatotimes B_1\hatotimes B_2))
\\\phi\hatotimes f_1\hatotimes f_2&\mapsto \overline{t\mapsto (V_1\hatotimes V_2)^*\phi(t^{-1}(D_1\times D_2))(f_1\hatotimes f_2)(V_1\hatotimes V_2)}\,.
\end{align*}
The equality of the two asymptotic morphisms can now be checked directly on the generating functions $u(x)=e^{-x^2}$ and $v(x)=xe^{-x^2}$ of $\cS$. Like in the proof of \cite[Theorem 4.1]{Dumitrascu}, this follows from the formula of Lemma \ref{lem:externalproductfunctionalcalculus}.
\end{proof}

In the following theorem, for any interval $I$ we let $S_I$ denote the trivial Dirac bundle with fiber $\Cl_1$ over $I$ and  $D_I$ denote its Dirac operator. For $I=\R$, this operator is already known from the end of Section \ref{sec:Etheory}.
\begin{thm}[Bordism invariance]\label{thm:BordismInvariance}
Let $W$ be a complete Riemannian manifold with a boundary $\partial W$ which decomposes into two complete Riemanian manifolds $M_0$, $M_1$ and assume that there are collar neighborhoods $M_0\times[0,2\varepsilon)$ of $M_0$ and $M_1\times(-2\varepsilon,0]$ of $M_1$ on which the Riemannian metric of $W$ is the product metric.
Let $S_{0}\to M_0$ and $S_1\to M_1$ be Dirac $A$-bundles with associated Dirac operators $D_0$, $D_1$, respectively, and assume that there is a Dirac $A\hatotimes\Cl_1$-bundle $S\to W$, which restricts to the product Dirac $A\hatotimes\Cl_1$-bundles $S_0\hatboxtimes S_{[0,2\varepsilon)}$ and $S_1\hatboxtimes S_{(-2\varepsilon,0]}$ over the respective collar neighborhoods. 
Then for every graded unital $C^*$-algebra $B$ the following diagram in $E$-theory commutes:
\[\xymatrix@C=3em{
&\cA(M_0,B)\ar[r]_-{\llbracket D_0;B\rrbracket}&C^*(M_0,A\hatotimes B)\ar[dr]_{(M_0\subseteq W)_*}&
\\\cA(W,B)\ar[ur]_{\text{restr.}}\ar[dr]^{\text{restr.}}
&&&C^*(W,A\hatotimes B)
\\&\cA(M_1,B)\ar[r]^-{\llbracket D_1;B\rrbracket}&C^*(M_1,A\hatotimes B)\ar[ur]^{(M_1\subseteq W)_*}&
}\]
\end{thm}
\begin{proof}
Let us first consider the upper composition using the diagram in Figure \ref{fig:bordinv}.
\begin{sidewaysfigure}
\[\xymatrix@R=1.5cm@C=3.5cm{
\cA(W,B)\ar[d]^{\text{restr.}}&&
\\\cA(M_0,B)\ar[r]^{\llbracket D_0;B\rrbracket}\ar[d]^{\text{Bott}}
&C^*(M_0,S_0\hatotimes B)\ar[d]^{\cong} \ar[dr]^{\id_*}&
\\\cA(M_0,B)\hatotimes C_0(0,\varepsilon)\hatotimes\Cl_1 \ar[r]^{\llbracket D_0;B\rrbracket\hatotimes\llbracket D_{(0,\varepsilon)};\C\rrbracket\hatotimes\id_{\Cl_1}}\ar[d]^{\text{incl.}}
&C^*(M_0, S_0\hatotimes B)\hatotimes\Cl_2
\ar[r]^{\alpha}\ar[d]_{\beta}
&C^*(M_0,A\hatotimes B)\ar[ddd]_{(M_0\subseteq W)_*}
\\\cA(M_0\times(0,\varepsilon),B)\hatotimes \Cl_1\ar[r]^{\llbracket D_0\times D_{(0,\varepsilon)};B\rrbracket\hatotimes\id_{\Cl_1}}\ar[dr]_{\llbracket D_{M_0\times (0,\varepsilon)};B\rrbracket\hatotimes\id_{\Cl_1}\qquad} \ar[dd]^{\text{incl.}}
&{\begin{array}{l}C^*((M_0\times[0,\varepsilon])^{\subseteq M_0\times\R},
\\\phantom{C^*(}(S_0\hatboxtimes S_{[0,\varepsilon]})\hatotimes B)\hatotimes\Cl_1\end{array}}\ar[d]^{\id_*}
&
\\&{\begin{array}{l}C^*((M_0\times[0,\varepsilon])^{\subseteq W},
\\\phantom{C^*(}S|_{M_0\times[0,\varepsilon]}\hatotimes B)\hatotimes\Cl_1\end{array}}\ar[d]^{(M_0\times [0,\varepsilon]\subseteq W)_*}
&
\\\cA(\mathring W,B)\hatotimes\Cl_1\ar[r]^{\llbracket D_{\mathring W};B\rrbracket\hatotimes\id_{\Cl_1}}
&C^*(W,S\hatotimes B)\hatotimes\Cl_1\ar[r]^{\gamma}
&C^*(W,A\hatotimes B)
}\]
\caption{Proving bordism invariance.}
\label{fig:bordinv}
\end{sidewaysfigure}
The arrows labeled by ``incl.'' are the canonical inclusions of $C^*$-algebras, ``restr.'' denotes restriction of functions and ``Bott'' denotes the isomorphism given by exterior product with the Bott element $\mathfrak{b}$ composed by the inclusion $C_0(0,\varepsilon)\subseteq C_0(\R)$, which is an isomorphism in $E$-theory. 
The interior of $W$ is denoted by $\mathring W$ and the Dirac operator of $S|_{\mathring W}$ by $D_{\mathring W}$.

The three arrows marked by $\alpha$, $\beta$, $\gamma$ are functoriality under the coarse maps $\id_{M_0}$, $M_0\xrightarrow{\subseteq} M_0\times [0,\varepsilon]$, $\id_W$, respectively, \`a la Lemma \ref{lem:coarseFunctoriality}, possibly combined with some isomorphism of Lemma \ref{lem:RoeClifford} and or the canonical isomorphism in $E$-theory $\Cl_2\cong M_2\cong \C$.
An equivalent description of $\beta$ is to combine the inclusion 
\[\Cl_1\subseteq \K_{\Cl_1}(L^2([0,\varepsilon],S_{[0,\varepsilon]})\hatotimes\widehat\ell^2)=C^*([0,\varepsilon],S_{[0,\varepsilon]})\]
with the canonical exterior multiplication of Roe algebras seen in  Lemma~\ref{lem:extProd}.

We denoted by $D_{M_0\times(0,\varepsilon)}$ the restriction of $D_{\mathring W}$ to $M_0\times(0,\varepsilon)$. It is, by assumption, the same as $D_0\times D_{(0,\varepsilon)}$, but we use different notations in this diagram to indicate that the $E$-theory class of the first is taken with respect to the coarse structure induced by $W$ while the $E$-theory class of the second is taken with respect to the coarse structure induced by $M_0\times\R$.

The top left square in the diagram commutes because of Theorem \ref{thm:Bottperiodicity} and commutativity of
\[\xymatrix{
C_0(0,\varepsilon)\ar[rr]^{\llbracket D_{(0,\varepsilon)};\C\rrbracket} \ar[d]^{\text{incl.}}&&\C
\\C_0(\R)\ar[urr]_{\mathfrak{b}^*}&&
}\]
in $E$-theory, which works analogous to Lemma \ref{lem:restriction}.

The remaining parts of the left side of the diagram commute because of (from top to bottom) Lemmas \ref{lem:extProd}, \ref{lem:independence} and \ref{lem:restriction}.
The right part commutes because of the naturality of the isomorphisms of Lemma \ref{lem:RoeClifford}.

It follows that the composition of the upper arrows in the assertion is the same as the composition of $\llbracket D_{\mathring W};B\rrbracket\hatotimes 1_{\Cl_1}$ with
\begin{align}
\cA(W,B)&\xrightarrow{\text{restr.}}\cA(M_0,B)\nonumber
\\&\xrightarrow{\text{Bott}}\cA(M_0,B)\hatotimes C_0(0,\varepsilon)\hatotimes\Cl_1\nonumber
\\&\xrightarrow{\text{incl.}}\cA(\mathring W,B)\hatotimes\Cl_1\,.\label{eq:uppercompBordism}
\end{align}
Analogously one sees that the composition of the lower arrows is the same as the composition of 
$\llbracket D_{\mathring W};B\rrbracket\hatotimes 1_{\Cl_1}$ with
\begin{align}
\cA(W,B)&\xrightarrow{\text{restr.}}\cA(M_1,B)\nonumber
\\&\xrightarrow{\text{Bott}}\cA(M_0,B)\hatotimes C_0(-\varepsilon,0)\hatotimes\Cl_1\nonumber
\\&\xrightarrow{\text{incl.}}\cA(\mathring W,B)\hatotimes\Cl_1\,.\label{eq:lowercompBordism}
\end{align}
It therefore suffices to show that the difference in $E$-theory of \eqref{eq:uppercompBordism} and \eqref{eq:lowercompBordism} vanishes.

Recall that the reflection map $s\colon \R\to\R,x\mapsto -x$ induces $s^*=-\id\in E(C_0\R,C_0\R)$. From this one easily deduces that 
\[\xymatrix{
\C\ar[rr]^-{\text{Bott}}\ar[drr]_-{\text{Bott}}&&C_0(0,\varepsilon)\hatotimes\Cl_1\ar[d]^{s^*}
\\&&C_0(-\varepsilon,0)\hatotimes\Cl_1
}\]
commutes up to the sign $-1$. This will take care of the minus sign.

As the Bott isomorphism also commutes with the restriction homomorphisms, it suffices to show that the sum  in $E$-theory of 
\begin{align*}
\kappa_0\colon \cA(W,B)\hatotimes C_0(0,\varepsilon)&\xrightarrow{\text{restr.}}\cA(M_0,B)\hatotimes C_0(0,\varepsilon)
\xrightarrow{\text{incl.}}\cA(\mathring W,B)
\end{align*}
and 
\begin{align*}
\kappa_1\colon \cA(W,B)\hatotimes C_0(0,\varepsilon)&\xrightarrow{\text{restr.}}\cA(M_1,B)\hatotimes C_0(0,\varepsilon)
\\&\xrightarrow{s^*}\cA(M_1,B)\hatotimes C_0(-\varepsilon,0)
\xrightarrow{\text{incl.}}\cA(\mathring W,B)
\end{align*}
vanishes. But as the images of these $*$-homomorphisms are orthogonal, their sum $\kappa=\kappa_0+\kappa_1$  is again a $*$-homomorphism and the $E$-theory class of this sum is the same as the sum of their $E$-theory classes. This follows from a simple homotopy.

This sum, however, can easily be described as follows:
Let $\pi\colon \partial W\times [0,\varepsilon]\to \partial W$ be the canonical projection and $\dist_{\partial W}\colon W\to \R$ be the function returning the distance to the boundary. We consider every $\psi\in C_0(0,\varepsilon)$ to be extended to all of $\R$ by $0$.
Then $\kappa$ maps $f\hatotimes\psi\in\cA(W,B)\hatotimes C_0(0,\varepsilon)$ to $(f|_{\partial W}\circ\pi)\cdot(\psi\circ\dist_{\partial W})$. This $*$-homomorphism is homotopic in the obvious way to $f\hatotimes\psi\mapsto f\cdot (\psi\circ\dist_{\partial W})$, because $\grad(f)$ is bounded, and the latter is homotopic via $f\hatotimes\psi\mapsto f\cdot \psi(t\cdot\dist_{\partial W})$ ($t\in[0,1]$) to the zero homomorphism.

Thus, $0=\llbracket\kappa\rrbracket =\llbracket\kappa_0\rrbracket + \llbracket\kappa_1\rrbracket$
and the proof is complete.
\end{proof}

This theorem is a far reaching enhancement of the bordism invariance of \cite{WulffBordism}. To obtain the bordism invariance of the coarse index from the bordism invariance of the $E$-theory class, just compose with the canonical element $1_W\in K(\cA(W,\C))$. In the present context we only consider Dirac-type operators, but on the other hand we allow more general coefficient $C^*$-algebras $A$ than just $\C$.

In the following we shall only need a few special cases, where $W$ is a product bordism.

\begin{cor}\label{cor:princSymp}
The $E$-theory class $\llbracket D;B\rrbracket\in E(\cA(M,B),C^*(M,A\hatotimes B))$ depends only on the principle symbol of $D$, i.\,e.\ it is independent of the choice of the Dirac connection on the bundle $S$.
\end{cor}
\begin{proof}
Given two Dirac connections $\nabla_0$, $\nabla_1$ on $S$, we consider the product bordism $W=M\times[0,1]$ with the product metric and the product Dirac bundle $S_W=S\hatboxtimes S_{[0,1]}$ but the connection we choose on $S_W$ interpolates between the product connections of $\nabla_0$, $\nabla_1$ with the canonical connection on $S_{[0,1]}$.

In this case, the inclusions of $M$ into $W$ as slices all induce the same canonical isomorphism $C^*(M,A\hatotimes B)\cong C^*(W,A\hatotimes B)$ and the compositions \[\cA(M,B)\xrightarrow[\text{projection}]{\text{pullback under}}\cA(W,B) \xrightarrow[\text{at } 0 \text{ or }1]{\text{restriction}}\cA(M,B)\] are the identity.
The theorem thus implies the equality of $\llbracket D_0;B\rrbracket$ and $\llbracket D_1;B\rrbracket$.
\end{proof}

The corollary allows us to prove the main theorem of this section:
\begin{thm}\label{mainthm:indexformula}
Let $M$ be a complete Riemannian manifold, $S\to M$ a Dirac $A$-bundle with associated Dirac operator $D$
and $E\to M$  a $B$-bundle of bounded variation.
Then the index of the twisted operator $D_E$ is
\[\ind(D_E)=\llbracket D;B\rrbracket\circ\llbracket E\rrbracket\,.\]
\end{thm}
\begin{proof}
By Corollary \ref{cor:princSymp}, the $E$- and $K$-theory classes in the assertion are independent of the choice of the connections on $S$ and $E$, so we are free to choose them as we like.
According to Definition \ref{def:bundleboundedvariation} the bundle $E$ was constructed from a projection $P\in\cA^\infty(M,B)$ and hence it already comes with a smooth adjointable isometric inclusion $W\colon E\xrightarrow{\subseteq}M\times B\hatotimes\widehat\ell^2$ with the property that $WW^*=P$.
We simply choose the connection induced by it (cf.\ proof of Lemma \ref{lem:twistedConnection}). On $S$ we choose the connection induced by any smooth adjointable isometric inclusions
$V\colon S\xrightarrow{\subseteq}M\times A\hatotimes\widehat\ell^2$.

The same Lemma \ref{lem:twistedConnection} applied to $S$ and the trivial rank-one bundle $M\times B\to M$ yields another Dirac $A\hatotimes B$-bundle $S\hatotimes B$ with Dirac operator $D_B$. The closure of $D_B$ is a regular selfadjoint operator on $L^2(S\hatotimes B)$, and consequently the direct sum of infinitely many copies of $D_B$ is a regular selfadjoint operator $D_B^\infty=D_B\hatotimes\id_{\widehat\ell^2}$ on $L^2(S\hatotimes B)\hatotimes\widehat\ell^2$.
It is not an $A\hatotimes B$-linear Dirac operator in our sense, because we require finitely generated Hilbert modules as fibres. Nonetheless, it behaves like a $A\hatotimes B$-linear Dirac operator whenever we need it.
We note that the construction in the proof of Lemma \ref{lem:twistedConnection} immediately implies $D_E=(\id\hatotimes W)^*\circ D_B^\infty\circ (\id\hatotimes W)$. 

On $L^2(S\hatotimes B)\hatotimes\widehat\ell^2$ we also define the self-adjoint unbounded operator 
\[D_P:=P\circ D_B^\infty\circ P+(\id-P)\circ D_B^\infty\circ (\id-P)\,.\]
The same calculation which proves Lemma \ref{lem:commutatorgradient} also shows
$[D_B^\infty,\id_S\hatotimes P]=\grad (P)\in\Gamma_b(\Cl(TM)\hatotimes B\hatotimes\widehat\K)$
and thus allows us to express $D_P$ as
\[D_P=D_B^\infty+(\id_S\hatotimes(2P-\id))\circ\grad (P)\,.\]
Therefore it is also regular by the following lemma.
\begin{lem}
Let $E,F$ be Hilbert $A$-modules, $t\colon E\to F$ a regular operator and $b\colon E\to F$ a bounded adjointable operator. Then $t+b$ is a regular operator, too.
\end{lem}
\begin{proof}
By \cite[Theorem 9.3]{Lance}, the graph $G(t)$ of $t$ is an orthocomplemented submodule of $E\oplus F$, so in particular the inclusion $i\colon G(t)\to E\oplus F$ is adjointable. The Hilbert module isomorphism 
\[j\colon E\oplus F\to E\oplus F\,,\quad (x,y)\mapsto (x,y+bx)\]
maps the closed subset $G(t)\subseteq  E\oplus F$ to $G(t+b)\subseteq  E\oplus F$. Therefore, $j\circ i$ is adjointable and its range $\im(j\circ i)=G(t+b)$ is closed. It now follows from \cite[Theorem 3.2]{Lance} that it is orthocomplemented and \cite[Proposition 9.5]{Lance} implies that $t+b$ is regular.
\end{proof}
We may now apply functional calculus and for $\phi_\pm(x)=(x\pm i)^{-1}$ we see that
\begin{align*}
\phi_\pm(t^{-1}D_B^\infty)-&\phi_\pm(t^{-1}D_P)=
\\&=t^{-1}\phi_\pm(t^{-1}D_B^\infty)\circ(\id_S\hatotimes(2P-\id)) \circ\grad (P) \circ\phi_\pm(t^{-1}D_P)
\end{align*}
converges to $0$ as $t\to\infty$, because $\grad (P)$ is bounded. This implies that the two asymptotic morphisms
\begin{align*}
\cS&\to\fA(\B_{A\hatotimes B}(L^2(S\hatotimes B)\hatotimes\widehat\ell^2)\,,
\\\phi&\mapsto\phi(t^{-1}D_B^\infty)\qquad\text{and} \qquad\phi\mapsto\phi(t^{-1}D_P)
\end{align*}
are in fact equal.

Note that the operator $D_P$ preserves the submodules $\im(P)$ and $\im(\id-P)$. Thus, it decomposes into the direct sum of two regular operator on these submodules.  This allows us to calculate
\begin{align*}
\phi(t^{-1}D_E)&=\phi(t^{-1}(\id\hatotimes W)^*\circ D_B^\infty\circ (\id\hatotimes W))
\\&=\phi(t^{-1}(\id\hatotimes W)^*\circ D_P\circ (\id\hatotimes W))
\\&=(\id\hatotimes W)^*\circ\phi(t^{-1} D_P)\circ (\id\hatotimes W)
\\&\varpropto(\id\hatotimes W)^*\circ\phi(t^{-1} D_B^{\infty})\circ (\id\hatotimes W)
\\&=(\id\hatotimes W)^*\circ(\phi(t^{-1} D_B)\hatotimes\id_{\widehat\ell^2})\circ (\id\hatotimes W)
\\&=(\id\hatotimes W)^*\circ(\phi(t^{-1} D)\hatotimes\id_{B\hatotimes\widehat\ell^2}) \circ (\id\hatotimes W)
\end{align*}
and consequently
\begin{align*}
(V\hatotimes W)&\circ\phi(t^{-1}D_E)\circ (V\hatotimes W)^*\varpropto
\\&\varpropto(V\hatotimes P)\circ (\phi(t^{-1}D)\hatotimes\id_{B\hatotimes\widehat\ell^2}) \circ(V^*\hatotimes P)
\\&\varpropto(V\phi(t^{-1}D)V^* \hatotimes \id_{B\hatotimes\widehat\ell^2})\circ(\id\hatotimes P)
\end{align*}

The claim now follows by recalling that the index of $D_E$ is represented by the asymptotic morphism
\begin{align*}
\cS\to&\fA(C^*(M,A\hatotimes B))
\\\phi\mapsto&\overline{t\mapsto(V\hatotimes W)^*\circ\phi(t^{-1}D_E)\circ (V\hatotimes W)}
\\&=\overline{t\mapsto(V^*\phi(t^{-1}D)V\hatotimes \id_{B\hatotimes\widehat\ell^2})\circ(\id\hatotimes P)}
\end{align*}
and this is exactly the canonical representative of the product of the $E$-theory class $\llbracket D;B\rrbracket\in E(\cA(M,B),C^*(M,A\hatotimes B))$ and the $K$-theory class $\llbracket E\rrbracket\in K(\cA(M,B))$.
\end{proof}

\section{Relation of the $E$-theory classes to $K$-homology classes}\label{sec:descent}
Recall from \cite[Chapter 12]{HigRoe} that the coarse index $\ind(D)\in K_n(C^*(M))$ of a $\Cl_n$-linear Dirac operator $D$ can be computed from its fundamental class $[D]\in K_n(M)$  by applying the coarse assembly map. A natural question is in how far our  $E$-theoretic fundamental class $\llbracket D;B\rrbracket$ is already determined by this $K$-homological fundamental class.
In this section we construct a ``descent'' homomorphism
\begin{equation}\label{eq:descentmap}
K_n(M)\to E(\cA(M,B),C^*(M,\Cl_n\hatotimes B))
\end{equation}
for manifolds $M$ of bounded geometry which maps
$[D]$ to $\llbracket D;B\rrbracket$. 
Its construction is motivated by the observation that
the formula in \cite{ZeidlerPSCProductSecondary} describing the $K$-homology class by means of Yu's localization algebra (cf.\ Definition \ref{def:localizationFundamentalclass} below) is very similar to our definition of the $E$-theory class.

It has to be assumed that there are similarly more general descent homomorphisms
\[E(C_0(M),A)\to E(\cA(M,B),C^*(M,A\hatotimes B))\]
for $A$-linear Dirac operators, where the fundamental class on the left hand side is defined by an asymptotic morphism  $\cS\hatotimes C_0(M)\to \fA(\K(L^2(S)))\subseteq \fA(A\hatotimes \widehat\K)$ similar to the one of Lemma \ref{lem:Eclassasymptoticmorphism}.
This general form will not be pursued in this paper.

To construct \eqref{eq:descentmap}, let $M$ be a complete Riemannian manifold of bounded geometry (cf.\ {\cite[Definition 2.1]{RoeIndexOpenManifolds}}). We briefly review this notion and its basic properties.
\begin{defn}\label{def:boundedGeometry}
\begin{itemize}
\item A complete Riemannian manifold $M$ has \emph{bounded geometry} if 
it has uniformly positive injectivity radius, 
i.\,e.\ $r_0:=\inf\{r_{inj}(x)\mid x\in M\}>0$, and
the curvature tensor $R$ of $M$ and all its derivatives $\nabla^kR$, $k\in\N$, are uniformly bounded.
\item A metric space $X$ has \emph{bounded geometry} if there is a subset $Y\subseteq X$ and a number $r>0$ such that $X=\bigcup_{y\in Y}B_r(y)$ and for each $R>0$ the number $\#(Y\cap B_R(x))$ is uniformly bounded in $x$ by some constant $K_{r,R}>0$.
\end{itemize}
\end{defn}
It is a well known fact that Riemannian manifolds of bounded geometry do also have bounded geometry when seen as metric spaces. This also follows from the following lemma.
\begin{lem}\label{lem:BGmainproperties}
Let $M$ be a complete Riemannian manifold of bounded geometry. For each $r>0$, let $\mathcal{Y}_r$ be the set of subsets $Y\subseteq M$ such that $\{B_r(y)\}_{y\in Y}$ is a cover of $M$ and $d(y,z)\geq r$ for all $y,z\in Y$, $y\not=z$. 
For each $R>0$ define
\[K_{r,R}:=\sup\{\#(Y\cap B_R(x))\mid Y\in\mathcal{Y}_r,x\in M\}\,.\]
Then $K_{r,R}$ is finite for all $r,R>0$ and for all $\alpha>0$ there is $\varepsilon>0$ such that $K_{r,\alpha r}$ is bounded uniformly in $r\in(0,\varepsilon)$.
\end{lem}
We remark that each $\mathcal{Y}_r$ is non-empty by Zorn's lemma and all $Y\in\mathcal{Y}_r$ are countable, because manifolds are assumed to be second countable by definition. 
The proof of the Lemma is a modification of the proof of \cite[Lemma 1.2]{Shubin}:
\begin{proof}
Let $\alpha>0$. For all $r>0$ such that $\alpha r+\frac r2<r_0$ and all $y\in Y\cap B_{\alpha r}(x)$, the balls $B_{r/2}(y)$ are contained in  $B_{\alpha r+r/2}(x)$ and disjoint for distinct $y$. Thus,
\[\#(Y\cap B_{\alpha r}(x))\cdot\inf_{z\in X}\operatorname{Vol}(B_{r/2}(z))\leq \sup_{z\in X}\operatorname{Vol}(B_{\alpha r+r/2}(z))\]
and we conclude from
\[K_{r,\alpha r}\leq \frac{\sup_{z\in X}\operatorname{Vol}(B_{\alpha r+r/2}(z))}{\inf_{z\in X}\operatorname{Vol}(B_{r/2}(z))}\xrightarrow{r\to 0}(2\alpha+1)^{\dim(M)}\]
that $K_{r,\alpha r}$ is bounded by a constant slightly larger than $(2\alpha+1)^{\dim(M)}$ for small $r$.

To prove the other parts of the statement, we use the finiteness of $K_{r,3r}$ for small $r>0$ in the following way: For $R>0$ define $N:=\lceil \frac{R}{r}\rceil$. If $x\in X$ and $y\in Y\cap B_R(x)$, choose  a minimal geodesic $\gamma\colon [0,1]\to X$ from $x$ to $y$ and pick
$y_i\in Y\cap B_r(\gamma(\frac{i}{N}))$ for $i=1,\dots,N-1$. Then $d(x,y_1)<2r$, $d(y_i,y_{i+1})<3r$ for $i=1,\dots,N-2$ and $d(y_{N-1},y)<2r$. By counting consecutively the possibilities for $y_1,\dots,y_{N-1},y$ for each fixed $x$, we obtain the inequality
$K_{r,R}\leq (K_{r,3r})^{N}<\infty$.

The proof is completed by the following monotony properties: If $0<r_1<r_2$ and $R>0$, then any $Y_2\in \mathcal{Y}_{r_2}$ is contained in some $Y_1\in \mathcal{Y}_{r_1}$ (again by Zorn's lemma) and therefore $K_{r_2,R}\leq K_{r_1,R}$.
\end{proof}
Later on we will also need the following version of \cite[Lemmas 1.3]{Shubin}:
\begin{lem}\label{lem:BGPartitionofUnity}
Let $M$ be a complete Riemannian manifold of bounded geometry. Then there exists an $r>0$, a $Y\in\mathcal{Y}_r$ and a partition of unity $\{\psi_y\}_{y\in Y}$ with the properties that $\supp(\psi_y)\subseteq B_r(y)$ and such that the gradients $\grad(\psi_y)$ are uniformly bounded in norm by some constant $C$. \qed
\end{lem}
We construct the homomorphism \eqref{eq:descentmap} by expressing the $K$-homology of $M$ as the $K$-theory of Yu's localization algebra from \cite{YuLocalization}, see also \cite{QiaoRoe}. We use the $\Cl_n$-linear version of  this algebra from \cite[Definition 3.1(3)]{ZeidlerPSCProductSecondary}:
\begin{defn}\label{def:LocalizationAlgebra}
The $\Cl_n$-linear \emph{localization algebra} $C^*_L(M,\Cl_n)$ is the $C^*$-subalgebra of $C_b(\,[1,\infty)\,,\,C^*(M,\Cl_n))$ generated by the uniformly continuous functions $L\colon [1,\infty)\to C^*(M;\Cl_n)$ such that the propagation of $L(t)$ is finite for all $t$ and tends to zero as $t\to\infty$.
\end{defn}
\begin{lem}[{\cite{QiaoRoe}}]\label{lem:localizationduality}
$K_n(M)\cong K(C^*_L(M,\Cl_n))$ \qed
\end{lem}
\begin{lem}
If $M$ is a complete Riemannian manifold of bounded geometry, then there is an asymptotic homomorphism
\begin{align*}
m\colon C^*_L(M,\Cl_n)\hatotimes \cA(M,B)&\to \fA(C^*(M,\Cl_n\hatotimes B))
\\L\hatotimes f&\mapsto\overline{t\mapsto L_tf}
\end{align*}
\end{lem}
Recall that overlines in this paper always denote equivalence classes.
\begin{proof}
We have to prove that the graded commutators $[L_t,f]$ vanish in the limit $t\to\infty$ for generating elements $L$ with propagation $p_t:=\operatorname{prop}(L_t)\xrightarrow{t\to\infty}0$ as in Definition \ref{def:LocalizationAlgebra} and  $f\in\cA^\infty(M,B)$.

For $r>0$ choose $Y\in\mathcal{Y}_r$ and a Borel decomposition $\{Z_y\}_{y\in Y}$ subordinate to $\{B_r(y)\}_{y\in Y}$. Denote by $\chi_y$ the characteristic function of $Z_y$ and define $f_r:=\sum_{y\in Y}f(y)\chi_y$. 
Thus
$\|f_r-f\|\leq \|\grad(f)\|\cdot r$. Furthermore, 
\begin{align*}
[L_t,f_r]&=L_t\sum_{z\in Y}f(z)\chi_z-(-1)^{\partial L\partial f}\sum_{y\in Y}f(y)\chi_yL_t
\\&=\sum_{y,z\in Y}\chi_yL_t\chi_z(f(z)-f(y))
=\sum_{\substack{y,z\in Y\\d(y,z)\leq 2r+p_t}}\chi_yL_t\chi_z(f(z)-f(y))
\end{align*}
with sums converging in the strong topology
and for $v\in L^2(M,\Cl_n\hatotimes B\hatotimes\widehat\ell^2)$, $v_z:=\chi_zv$ we estimate the norm
\begin{align*}
\|[L_t,f_r]v\|^2&=\sum_{y\in Y}\left\|\sum_{\substack{z\in Y\\d(y,z)\leq 2r+p_t}}\chi_yL_t(f(z)-f(y))v_z\right\|^2
\\&\leq\sum_{y\in Y}\left(\sum_{\substack{z\in Y\\d(y,z)\leq 2r+p_t}}\|L_t\|\cdot \|\grad(f)\|\cdot(2r+p_t)\cdot\|v_z\|\right)^2
\\&\leq\sum_{y\in Y}K_{r,2r+p_t}\cdot\sum_{\substack{z\in Y\\d(y,z)\leq 2r+p_t}}\|L_t\|^2\cdot \|\grad(f)\|^2\cdot(2r+p_t)^2\cdot\|v_z\|^2
\\&\leq K_{r,2r+p_t}^2\cdot\|L_t\|^2\cdot \|\grad(f)\|^2\cdot(2r+p_t)^2\cdot\|v\|^2
\end{align*}
where the second inequality was the inequality between arithmetic and quadratic mean.
Thus, $\|[L_t,f_r]\|\leq  K_{r,2r+p_t}\cdot\|L_t\|\cdot \|\grad(f)\|\cdot(2r+p_t)$ and therefore
\[\|[L_t,f]\|\leq  (K_{r,2r+p_t}+1)\cdot\|L_t\|\cdot \|\grad(f)\|\cdot(2r+p_t)\,.\]
By letting $r$ go to $0$ and using the boundedness of $K_{r,\alpha r}$ for small $r$ we conclude, that the commutator vanishes in the limit $t\to\infty$.
\end{proof}

The homomorphism \eqref{eq:descentmap} is now simply obtained from the previous two lemmas using an $E$-theoretic product:
\begin{defn}
The \emph{descent homomorphism}
\[\nu\colon K_n(M)\to E(\cA(M,B),C^*(M,\Cl_n\hatotimes B))\]
maps an element $x\in K_n(M)\cong K(C_L^*(M,\Cl_n))$ to 
the product
$\llbracket m\rrbracket\circ(x\otimes 1_{\cA(M,B)})$.
\end{defn}
It obviously maps the following $K$-homology class of $D$ to the $E$-theory class of the previous section. 
\begin{defn}[{cf.\ \cite[Definition 4.1]{ZeidlerPSCProductSecondary}}] \label{def:localizationFundamentalclass}
Let $D$ be a $\Cl_n$-linear Dirac operator over $M$.
Under the isomorphism of Lemma \ref{lem:localizationduality}, the $K$-homology class $[D]\in K_n(M)$ is defined as corresponding to the $K$-theory element determined by the $*$-homomorphism
\[\cS\to C^*_L(M,\Cl_n)\,,\quad \phi\mapsto (t\mapsto \phi(t^{-1}D))\,.\]
\end{defn}

\begin{rem}
In \cite[Section 4.3]{AlexUniform} a cap product
\[K_p^u(M)\otimes K^q_u(M)\to K_{p-q}^u(M)\]
between the uniform $K$-homology and the uniform $K$-theory of a manifold of bounded geometry $M$ was constructed. 
It has the property that it maps the uniform $K$-homology class $[D]$ of a Dirac operator $D$ over $M$ and the uniform $K$-theory class $[E]$ of a vector bundle $E\to M$ of bounded geometry to the uniform $K$-homology class of the twisted operator: $[D]\cap[E]=[D_E]$.

In view of this formula and Theorem \ref{mainthm:indexformula}, we conjecture that the assembly map $\mu$, the descent homomorphism $\nu$, the forgetful map $\mathfrak{f}\colon K_*^u(M)\to K_*(M)$ and the canonical stabilization map $\iota\colon K^*_u(M)\to K_{-*}(\cA(M,\C))$ induced by inclusion of $C^*$-algebras are related by the formula.
\[(\nu(\mathfrak{f}(x))\circ\iota(y)=\mu(\mathfrak{f}(x\cap y))\]
for all $x\in K_*^u(M)$, $y\in K_u^*(M)$.
\end{rem}

\section{Twisting by more general vector bundles}\label{sec:generalbundles}
Let $S$ be a Dirac $A$-bundle over a complete Riemannian manifold $(M,g)$ with Dirac operator $D$. In this section we address the question what to do if we want to calculate the index of the twisted operator $D_E$ by a $B$-bundle $E$ which is represented by a smooth projection $P\in C^{\infty}(M,B\hatotimes\widehat\K)$ with \emph{unbounded gradient} $\grad(P)$.
In this case there exists a smooth function $\lambda\colon M\to [1,\infty)$ such that $P$ is a projection in the following algebra:
\begin{defn}
Denote by $\cA_\lambda(M,B)$ be the closure of
\[\cA_\lambda^\infty(M,B):=\{f\in C_b^\infty(M,B\hatotimes\widehat\K)\mid \lambda^{-1}\cdot\grad(f)\text{ is bounded}\}\]
in $C_b(M,B\hatotimes\widehat\K)$.
\end{defn}
We treat this case by a conformal change of the metric, which allows us to apply our previously developed theory. 
Let $M^\lambda$ denote the manifold $M$ equipped with the Riemannian metric $\lambda^2 g$. Note that it is again complete, because $\lambda\geq 1$, and that
\[\cA_\lambda^{(\infty)}(M,B)=\cA^{(\infty)}(M^\lambda,B)\,.\]

From $S$ we construct a Dirac $A$-bundle $S^\lambda\to M^\lambda$ whose underlying Hilbert $A$-module bundle is $S$, i.\,e.\ also the $A$-module multiplication and the $A$-valued scalar products on the fibers of $S$ and $S^\lambda$ are identical. It is obtained by rescaling the Clifford action of tangential vectors by the factor ${\lambda}$ and changing the connection $\nabla$ to 
\[\nabla^\lambda_X:=\nabla_X+\frac{1}{4\lambda}(\grad(\lambda) X-X \grad(\lambda))\,.\]
The second summand acts by Clifford multiplication.
It is straightforward to verify that this new connection is grading preserving, metric and fulfills the Leibniz rule with respect to the multiplication by $A$-valued smooth functions.

The Leibniz rule with respect to Clifford multiplication by smooth vector fields is more tricky to verify, as it contains traps: One has to keep in mind that both the Clifford action and the Levi--Civita connection are different for $M^\lambda$. The Levi--Civita connection $\nabla^{LC,\lambda}$ of $M^\lambda$ is related to the Levi--Civita connection of $M$ by the formula 
\[\nabla^{LC,\lambda}_XY=\nabla^{LC}_XY+\frac{1}{\lambda}\big((\partial_Y\lambda) X + (\partial_X\lambda)Y - g(X,Y)\grad(\lambda)\big)\]
as is easily derived from the Koszul formula. Hence, using the abbreviation $Z:=\grad(\lambda)$ we have
\begin{align*}
\lambda\nabla^{LC,\lambda}_X(Y)=&\lambda\nabla^{LC}_X(Y)+(\partial_Y\lambda) X + (\partial_X\lambda)Y - g(X,Y)Z
\\=&\nabla^{LC}_X(\lambda Y)+g(Y,Z)X - g(X,Y)Z
\\=&\nabla^{LC}_X(\lambda Y)+\frac14\big(-(YZ+ZY)X-X(YZ+ZY)
\\&\phantom{\nabla^{LC}_X(\lambda Y)+\frac14\big(}+(XY+YX)Z+Z(XY+YX)\big)
\\=&\nabla^{LC}_X(\lambda Y)+\frac14(YXZ+ZXY-YZX-XZY)
\end{align*}
and therefore
\begin{align*}
\nabla^\lambda_X(\lambda Y \xi)-&\lambda Y \nabla^\lambda_X\xi
=\nabla_X(\lambda Y \xi)+\frac{1}{4\lambda}(Z X-XZ)(\lambda Y \xi)
\\&\phantom{\lambda Y \nabla^\lambda_X\xi
=}-\lambda Y \left(\nabla_X\xi+\frac{1}{4\lambda}(Z X-XZ)\xi\right)
\\=&\left(\nabla^{LC}_X(\lambda Y)\right)\xi+\frac14(ZXY-XZY-YZX+YXZ)\xi
\\=&\lambda\left(\nabla^{LC,\lambda}_X(Y)\right)\xi\,,
\end{align*}
which is exactly the Leibniz formula for the new Clifford action with respect to the new connections.

From a given local orthonormal frame $e_1,\dots,e_{\dim M}$ of $TM$ we obtain the local orthonormal frame $\lambda^{-1}e_1,\dots,\lambda^{-1}e_{\dim M}$ of $TM^\lambda$. This gives us the relation between the Dirac operator $D$ of $S$ and the Dirac operator $D^\lambda$ of $S^\lambda$:
\begin{align*}
D^\lambda&=\sum_{i=1}^{\dim M}\lambda(\lambda^{-1}e_i)\nabla^{\lambda}_{\lambda^{-1}e_i}
\\&=\sum_{i=1}^{\dim M}\lambda^{-1}e_i\left(\nabla_{e_i}+\frac{1}{4\lambda}(\grad(\lambda)e_i-e_i\grad(\lambda))\right)
\\&=\frac{1}{{\lambda}}D+\frac{\dim M-1}{2\lambda^{2}}\grad(\lambda)
\end{align*}

\begin{thm}\label{thm:lambdaBundleTwist}
The index of the twisted operator $D_E$ is the $E$-theory product
\[\ind(D_E)=(c_\lambda)_*\circ \llbracket D^\lambda;B\rrbracket \circ \llbracket E\rrbracket\in K(C^*(M,A\hatotimes B))\]
of the element induced by the coarse map $c_\lambda:=\id\colon M^\lambda\to M$, the $E$-theory class $\llbracket D^\lambda;B\rrbracket\in E(\cA(M^\lambda,B),C^*(M^\lambda,A\hatotimes B))$ and the $K$-theory class $\llbracket E\rrbracket\in K(\cA_\lambda(M,B))$ determined by the projection $P\in\cA_\lambda(M,B)$.
\end{thm}

This is an easy consequence of the bordism invariance of the coarse index. In fact, it is a special case of the following Lemma:

\begin{lem}\label{lem:lambdaEtheoryclasses}
Let $\lambda_0, \lambda_1\colon M\to[1,\infty)$ be smooth functions with $\lambda_0\geq\lambda_1$, $c_{\lambda_0,\lambda_1}\colon M^{\lambda_0}\to M^{\lambda_1}$ the coarse map whose underlying map is the identity and $i_{\lambda_1,\lambda_0}\colon \cA_{\lambda_1}(M,B)\to \cA_{\lambda_0}(M,B)$ the inclusion. Then 
\[\llbracket D^{\lambda_1};B\rrbracket=(c_{\lambda_0,\lambda_1})_*\circ \llbracket D^{\lambda_0};B\rrbracket\circ i_{\lambda_1,\lambda_0}\,.\]
\end{lem}
\begin{proof}
Let $\omega\colon [0,1]\to \R$ be smooth, equal to $1$ near $0$ and equal to $0$ near $1$.
We equip $W=M\times [0,1]$ with the Riemannian metric 
$dt^2+\omega(t)\lambda_0^2g+(1-\omega(t))\lambda_1^2 g$.
Note that for $\varepsilon$ small enough the collar neighborhoods $M\times[0,2\varepsilon)$ and $M\times(1-2\varepsilon,1]$  carry the product metrics $dt^2+\lambda_0^2g$ and $dt^2+\lambda_1^2g$, respectively, and that the inclusion $\iota_1\colon M^{\lambda_1}= M\times\{1\}\subseteq W$ is a coarse equivalence.

Similar to the construction of $S^\lambda$ we can change the Clifford action and connection of $S\hatboxtimes S_{[0,1]}\to M\times[0,1]$ to obtain a Dirac $A\hatotimes\Cl_1$-bundle $S_W\to W$ which restricts to $S^{\lambda_0}\hatboxtimes S_{[0,2\varepsilon)}$ and $S^{\lambda_1}\hatboxtimes S_{(1-2\varepsilon,1]}$ over the respective collar neighborhoods of the boundary components.

By pullback under the projection $\pi\colon W\to M^{\lambda_1}$ we obtain a $*$-homo\-mor\-phism
$\pi^*\colon \cA_{\lambda_1}(M,B)\to\cA(W,B)$ which is left inverse to the restriction $*$-homomorphism $r_1\colon \cA(W,B)\to\cA_{\lambda_1}(M,B)$. Composing $\pi^*$ with the other restriction 
$r_0\colon \cA(W,B)\to\cA_{\lambda_0}(M,B)$ yields the inclusion
$i_{\lambda_1,\lambda_0}$.
The claim now follows from Theorem \ref{thm:BordismInvariance} by composing the diagram with $\pi^*$ from the left and with $((\iota_1)_*)^{-1}$ from the right.
\end{proof}
\begin{proof}[Proof of Theorem \ref{thm:lambdaBundleTwist}]
We apply this lemma to the Dirac $A\hatotimes B$-bundle $S\hatotimes B$ instead of the Dirac $A$-bundle $S$, $\C$ instead of $B$ and $\lambda_0=\lambda$, $\lambda_1=1$. By composing the resulting equation with $1_M\in K(\cA(M,\C))$ we obtain
\[\ind(D_E)=(c_\lambda)_*\ind((D_E)^\lambda)\,.\]
The proof is completed by noting that $(D_E)^\lambda=(D^\lambda)_E$, i.\,e.\ the operations of twisting and changing the metric conformally commute, and that the index of the latter is $\llbracket D^\lambda;B\rrbracket\circ \llbracket E\rrbracket$.
\end{proof}

Lemma \ref{lem:lambdaEtheoryclasses} allows us to construct an even more general ``$E$-theory class'' which incorporates all $\lambda$ at once. The set of smooth functions $\lambda\colon M\to [1,\infty)$ is a directed set under ``$\leq$'' and the $E$-theory groups $E(\cA_\lambda(M,B),C^*(M,A\hatotimes B))$
comprise a inverse system on this directed set via composing with the inclusions $i_{\lambda_1,\lambda_0}$.
\begin{defn}
The collection of elements 
\[(c_\lambda)_*\circ\llbracket D^\lambda;B\rrbracket\in E(\cA_\lambda(M,B),C^*(M,A\hatotimes B))\]
defines an element $\underleftarrow{\llbracket D;B\rrbracket}$ in the inverse limit
\[\underleftarrow{E}(C_b(M,B\hatotimes\widehat\K),C^*(M,A\hatotimes B)):=\varprojlim_\lambda E(\cA_\lambda(M,B),C^*(M,A\hatotimes B))\,.\]
\end{defn} 

The notation is explained as follows:
As the union of the $\cA_\lambda(M,B)$ is dense in $C_b(M,B\hatotimes\widehat\K)$ (it contains all smooth bounded functions), we have \[K(C_b(M,B\hatotimes\widehat\K))=\varinjlim_\lambda K(\cA_\lambda(M,B))\,.\]
Thus, there is a composition product
\[\circ\colon \underleftarrow{E}(C_b(M,B\hatotimes\widehat\K),C^*(M,A\hatotimes B))\hatotimes K(C_b(M,B\hatotimes\widehat\K))\to K(C^*(M,A\hatotimes B))\]
and Theorem \ref{thm:lambdaBundleTwist} implies:
\begin{thm}\label{thm:Inverselimitindexformula}
If the $B$-bundle $E$ is determined by a smooth projection valued function
$P\colon M\to B\hatotimes\widehat\K$, then 
$\ind(D_E)=\underleftarrow{\llbracket D;B\rrbracket}\circ \llbracket E\rrbracket$, where $\llbracket E\rrbracket$ denotes the $K$-theory class of $P$.
\end{thm}
This is the most general form of the index pairing.
\begin{rem}
It is somewhat awkward, that this $E$-theory class does not only depend on $C_b(M,B\hatotimes\widehat\K)$ and $C^*(M,A\hatotimes B)$ but also incorporates the directed system of $\lambda$'s. This raises the question whether there is an even more suitable version of bivariant $K$-theory, which does not show this inconvenience.
\end{rem}

\section{Dividing out the compact operators}\label{sec:dividingoutcompacts}
The coarse index was invented as a generalization of the $A$-index of elliptic operators over compact manifolds. Indeed, if $M$ is compact and connected\footnote{We have to make this assumption, because in our context non-connected compact manifolds are not bounded.}, and $S\to M$ a Dirac $A$-bundle with Dirac operator $D$, then $C^*(M,A)=\K(L^2(M)\hatotimes A\hatotimes\widehat{\ell^2})\cong A\hatotimes \widehat{\K}$ and the coarse index of $D$ corresponds to the $A$-index under the isomorphism $K(C^*(M,A))\cong K(A)$. However, the coarse index behaves quite different from the Fredholm index if $M$ is non-compact. For example, the following theorem is not true for compact $M$:
\begin{thm}[cf.\ {\cite[Theorem 1.7]{HankePapeSchick}}] \label{thm:HPSpartialvanishing}
Let $(M,g)$ be a complete connected non-compact Riemannian spin manifold such that, outside of a compact subset, the scalar curvature is uniformly positive. Let 
$E\to M$ be an $A$-bundle with a flat connection. Denote by $\slashed D$ the spin Dirac operator and by $\slashed D_E$ the twisted operator. Then the coarse index 
\[\ind(\slashed D_E)\in K_*(C^*(M,A))\]
vanishes.
\end{thm}
We can get rid of this distinction between compact and non-compact $M$ if we consider the coarse index as an element of the $K$-theory of the following quotient $C^*$-algebra instead.
\begin{defn}
For $M$ a complete Riemannian manifold with connected components $\{M_i\}_{i\in I}$ and $A$ a $\Z_2$-graded unital $C^*$-algebra we define the ideal
\[\K(M,A):=\bigoplus_{i\in I}\K(L^2(M_i)\hatotimes A\hatotimes\widehat{\ell^2}))=C^*(M,A)\cap\K(L^2(M)\hatotimes A\hatotimes\widehat{\ell^2})\]
of $C^*(M,A)$ and the quotient $C^*$-algebra
\[C^*_{/\K}(M,A):=C^*(M,A)/\K(M,A)=\bigoplus_{i\in I}C^*(M_i,A)/\K(M_i,A)\,.\]
\end{defn}

Indeed, if $M$ is compact, then this $C^*$-algebra vanishes and the theorem becomes trivially true. If, on the other hand, $M$ is non-compact and connected, then there is a coarsely embedded ray $[0,\infty)\to M$ and the map 
$K_*(\K(M,A))\to K_*(C^*(M,A))$ factors through $K_*(C^*([0,\infty),A))=0$ (cf.\ \cite[Proposition 3.10]{HankePapeSchick}).
Thus, the long exact sequence in $K$-theory decomposes into short exact sequences
\begin{equation}\label{eq:modKseq1}
0\to K_*(C^*(M,A))\to K_*(C^*_{/\K}(M,A))\to K_{*-1}(\K(M,A))\to 0
\end{equation}
and therefore no information is lost by passing to the quotient. If $M$ is non-compact but disconnected, then we can argue on each component separately to see that exactly the information of the compact components is lost.

Another way of thinking about the groups $K_*(C^*_{/\K}(M,A))$ is provided by Roe's $X^\to$-construction from \cite{RoeFoliations}: Given a proper metric space $X$, the proper metric space $X^\to$ is obtained from $X$ by attaching a ray $[0,\infty)$ to an arbitrary point of $X$. Applied to a \emph{non-compact, connected}, complete Riemannian manifold $M$, the coarse Mayer--Vietoris sequence (cf.\ \cite{HigsonRoeYuMayerVietoris} in the case of Roe algebras without coefficients, \cite[Corollary 9.5]{HigsonPedersenRoe} in the case with coefficients) yields short exact sequences
\begin{equation}\label{eq:modKseq2}
0\to K_*(C^*(M,A))\to K_*(C^*(M^{\to},A))\to K_{*-1}(\K(M,A))\to 0\,.
\end{equation}
Equations \eqref{eq:modKseq1} and \eqref{eq:modKseq2} and the five Lemma imply that the canonical comparision $*$-homomorphism $C^*(M^{\to},A)\to C^*_{/\K}(M,A)$, which we will see in more detail in Lemma \ref{lem:comparemodKandnotmodK} below, induces a canonical isomorphism
\[K_*(C^*(M^{\to},A))\cong K_*(C^*_{/\K}(M,A))\,.\]
Note that this trivially holds for compact $M$, too. 
Thus, dividing out the compact operators has the same effect of ``unreducing'' the ``coarse homology theory'' $K_*(C^*(M,A))$ as the constructions seen in \cite[Section 5]{WulffCoassemblyRinghomo}.

The coarse index in the $K$-theory of this quotient $C^*$-algebra, and similarly an $E$-theory class, can even be defined more generally for elliptic operators which are only defined on the complement of a compact subset, as we will now show.

\begin{lem}\label{lem:defWtildeMtildeStildeD}
Let $M$ be a complete Riemannian manifold, $K\subseteq M$ a compact subset and $S\to M\setminus K$ a Dirac $A$-bundle with Dirac operator $D$. Then there exist
\begin{itemize}
\item a complete submanifold $W\subseteq M$ of codimension $0$ with boundary such that  $M\setminus W$ is precompact and contains $K$,
\item a Riemann-isometric inclusion $W\subseteq \tilde M$  of $W$ as a complete submanifold of codimension $0$ with boundary into another complete Riemannian manifold $\tilde M$, such that the identity $W^{\subseteq\tilde M}\to W^{\subseteq M}$ is a coarse map and such that
\item $S|_W$ extends to a Dirac $A$-bundle $\tilde S\to\tilde M$ with Dirac operator $\tilde D$.\qed
\end{itemize}
\end{lem}
Such a manifold $\tilde M$ can be constructed essentially by doubling $W$, and so we observe that $\tilde M\setminus W$ need not be precompact and if $M$ was connected, then $\tilde M$ need not be connected. 

By assumption, 
\[W^{\subseteq \tilde M}\xrightarrow{\id}W^{\subseteq M}\xrightarrow{\text{incl.}}M\]
are coarse maps, and they are covered by the obvious isometries 
\[L^2(W,S|_W)\xrightarrow{\id}L^2(W,S|_W)\xrightarrow{\subseteq}L^2(M,S)\]
and thereby yielding canonical inclusions
\[C^*(W^{\subseteq\tilde M},S|_W\hatotimes B)\subseteq C^*(W^{\subseteq M},S|_W\hatotimes B)\subseteq C^*(M,S\hatotimes B)\,.\]
With these inclusions in mind, Lemma \ref{lem:relativeasymptoticmorphism} yields an asymptotic morphism
\begin{align*}
\cS\hatotimes\cA(\mathring W,B)&\to\fA(C^*(W^{\subseteq\tilde M},S|_W\hatotimes B))\subseteq\fA(C^*(M,A\hatotimes B))\,,
\\\phi\hatotimes f&\mapsto\overline{t\mapsto V^*\phi(t^{-1}\tilde D)fV}\,,%
\end{align*}
where $V\colon L^2(W,S|_W\hatotimes B) \hatotimes\widehat\ell^2 \to L^2(\tilde M,\tilde S\hatotimes B) \hatotimes\widehat\ell^2$ denotes the canonical isometry. The proof of Lemma \ref{lem:independence} shows that it is independent of the extension 
$\tilde M$ of $W$. Furthermore, it maps $\cS\hatotimes C_0(\mathring W,B\hatotimes\widehat\K)$ to $\K(M,A)$. Thus we obtain an asymptotic morphism
\begin{equation*}
\cS\hatotimes(\cA(\mathring W,B)/C_0(\mathring W,B\hatotimes\widehat\K))\to \fA(C^*_{/\K}(M,A\hatotimes B))\,.
\end{equation*}
Note that the quotient $\cA(\mathring W,A\hatotimes B)/C_0(\mathring W,B\hatotimes\widehat\K)$ depends only on $M$ and not on the choice of $W$, because it is isomorphic to the following $C^*$-algebra.
\begin{defn}
For a complete Riemannian manifold $M$ and a $\Z_2$-graded unital $C^*$-algebra $B$, let
\[\cA_{/C_0}(M,B):=\cA(M,B)/C_0(M,B\hatotimes\widehat\K)\,.\]
\end{defn}
The vector bundles defining elements in the $K$-theory of this $C^*$-algebras are the following:
\begin{defn}
Let $P\in\cA^\infty(M,B)$ be such that $P_x$ is a projection in $B\hatotimes\widehat\K$ for each $x$ outside of a compact subset $K\subseteq M$.
The $B$-bundle $E\to M\setminus K$ whose fibre at $x$ is the image of $P_x$ in $B\hatotimes\widehat\ell^2$ is called a \emph{$B$-bundle of bounded variation defined outside of a compact subset}.
Its $K$-theory class $\llbracket E\rrbracket_{/C_0}\in K(\cA_{/C_0}(M,B))$ is defined to be the class of the projection $\overline{P}\in\cA_{/C_0}(M,B)$ determined by $P$.
\end{defn}

Furthermore, the proof of Lemma \ref{lem:restriction} shows that the asymptotic morphism 
\begin{equation}\label{eq:modKasympmorph}
\cS\hatotimes\cA_{/C_0}(M,B)\to \fA(C^*_{/\K}(M,A\hatotimes B))
\end{equation}
so obtained is independent of the choice of $W$.
\begin{defn}\label{def:modKindex}
Let $D$ be as above an $A$-linear Dirac operator defined over the complement of a compact set $K\subset M$.
The $E$-theory class of $D$ is the class
\[\llbracket D;B\rrbracket_{/\K/C_0}\in E(\cA_{/C_0}(M,B),C^*_{/\K}(M,A\hatotimes B))\]
of the just defined asymptotic morphism \eqref{eq:modKasympmorph}. Its index is defined as
\[\ind_{/\K}(D):=\llbracket D;\C\rrbracket_{/\K/C_0}\circ 1_M\in K(C^*_{/\K}(M,A))\,.\] 
\end{defn}
Here, $1_M\in K(\cA_{/C_0}(M,\C))$ denotes the image of $1_M\in K(\cA(M,\C))$ under the quotient map.

Note that if $D$ is defined over all of $M$, then we can simply choose $K=\emptyset$, $\tilde M=W=M$, $\tilde S=S$ in the construction above to see that $\ind_{/\K}(D)$ is indeed the image of $\ind(D)$ under the canonical map.

The following lemmas provide an important alternative way of constructing the index modulo $\K$:
\begin{lem}\label{lem:comparemodKandnotmodK}
Let $M,K,S,D,W,\tilde M,\tilde S,\tilde D$ be as in the construction above and define $W^c:=\tilde M\setminus W$. 
\begin{itemize}
\item Let $\omega\in C_b(\tilde M)$ be $0$ on $W^c$ and $1$ outside of the $1$-neighborhood of $W^c$. Then 
\[\tau_{\tilde M,W,B}\colon \cA(\tilde M,B)\to \cA_{/C_0}(\mathring W,B)\subseteq\cA_{/C_0}(M,B)\,,\quad f\mapsto \omega f\]
is a $*$-homomorphism which is independent of the choice of $\omega$.
\item Denote by $V_{\tilde M,W,A}\colon L^2(W,A)\hatotimes\widehat\ell^2\to L^2(\tilde M,A)\hatotimes\widehat\ell^2$ the isometric inclusion. Then 
\begin{align*}
\pi_{\tilde M,W,A}\colon C^*(\tilde M,A)&\to C^*_{/\K}(W^{\subseteq\tilde M},A)\subseteq C^*_{/\K}(M,A)
\\T&\mapsto \overline{(V_{\tilde M,W,A})^*TV_{\tilde M,W,A}}
\end{align*}
is a $*$-homomorphism.
\item The following diagram commutes in $E$-theory:
\[\xymatrix@C=2cm{
\cA(\tilde M,B)\ar[r]^{\llbracket\tilde D;B\rrbracket}\ar[d]^{\tau_{\tilde M,W,B}}
&C^*(\tilde M,A\hatotimes B)\ar[d]^{\pi_{\tilde M,W,A\hatotimes B}}
\\\cA_{/C_0}(M,B)\ar[r]^{\llbracket D;B\rrbracket_{/\K/C_0}}
&C^*_{/\K}(M,A\hatotimes B)
}\]
\end{itemize}
\end{lem}
\begin{proof}
The first part is trivial.

For the second part, let $S,T$ have propagation bounded by $R>0$. Denote by $\chi\colon \tilde M\to\{0,1\}$ the characteristic function of $W$ and by 
$\rho\colon \tilde M\to\{0,1\}$ the characteristic function of the $R+1$-neighborhood of $W$ and abbreviate $V:=V_{\tilde M,W,A}$.
We make use of the propagation of $S,T$ to see that
\[V^*STV-V^*SVV^*TV=V^*S(1-\chi)TV=V^* S(1-\chi)\rho TV\,,\]
which is $A$-compact, because $\rho(1-\chi)$ has compact support. This proves multiplicativity.

The two compositions in the diagram  of the third part are represented by the asymptotic morphisms
\begin{align*}
\cS\hatotimes\cA(\tilde M,B)&\to\fA(C^*_{/\K}(M,A\hatotimes B))
\\\phi\hatotimes f&\mapsto\overline{t\mapsto \overline{(V_{\tilde M,W,A\hatotimes B})^*\,\phi(t^{-1}\tilde D)f\,V_{\tilde M,W,A\hatotimes B}}}\qquad\text{and}
\\\phi\hatotimes f&\mapsto\overline{t\mapsto \overline{(V_{\tilde M,W,A\hatotimes B})^*\,\phi(t^{-1}\tilde D)(\omega f)\,V_{\tilde M,W,A\hatotimes B}}}\,,
\end{align*}
respectively. The task of keeping a clear head about which operators in these formulas act on which Hilbert modules in which canonical  way is left to the reader.

The two asymptotic morphisms are equal, because the function $f(1-\omega)\chi$ has compact support and hence the difference 
\begin{align*}
(V_{\tilde M,W,A\hatotimes B})^*&\phi(t^{-1}\tilde D)fV_{\tilde M,W,A\hatotimes B}-(V_{\tilde M,W,A})^*\phi(t^{-1}\tilde D)V_{\tilde M,W,A}\,\omega f
\\&=(V_{\tilde M,W,A\hatotimes B})^*\phi(t^{-1}\tilde D)f(1-\omega)\chi V_{\tilde M,W,A\hatotimes B}
\end{align*}
is compact.
\end{proof}

\begin{cor}\label{cor:indexprojection}
$\pi_{\tilde M,W,A}\circ\ind(\tilde D)=\ind_{/\K}(D)$.
\end{cor}
\begin{proof}
The homomorphism $(\tau_{\tilde M,W,\C})_*\colon K(\cA(\tilde M,\C))\to K(\cA_{/C_0}(M,\C))$ maps $1_{\tilde M}$ to $1_{M}$. The claim now follows from the third part of the lemma with $B=\C$.
\end{proof}

Using this construction, we may now generalize Theorem \ref{thm:HPSpartialvanishing} to the situation where $M$ is only spin outside of the compact set $K$.
\begin{thm}
Let $(M,g)$ be a complete connected non-compact Riemannian manifold such that, outside of a compact subset $K$, it is spin and the scalar curvature is uniformly positive. Let 
$E\to M\setminus K$ be an $A$-bundle with a flat connection. Denote by $\slashed D$ the spin Dirac operator over $M\setminus K$ and by $\slashed D_E$ the twisted operator. Then the coarse index 
\[\ind_{/\K}(\slashed D_E)\in K_*(C_{/\K}^*(M,A))\]
vanishes. 
\end{thm}
\begin{proof}
In the construction of the index modulo $\K$ we can choose $W,\tilde M$ in such a way that $\tilde M$ is --as differentiable manifold-- the double of $W$ and the metric on $\tilde M\setminus W$ can be chosen such that $\tilde M$ has uniformly positive scalar curvature. Theorem \ref{thm:HPSpartialvanishing} implies $\ind(\widetilde{\slashed D_E})=0$ and the claim follows from the previous corollary.
\end{proof}

Another consequence of Corollary \ref{cor:indexprojection} is the expected formula for twisted operators:
\begin{thm}\label{mainthm:modKindexformula}
Let $M$ be a complete Riemannian manifold, $S\to M\setminus K$ a Dirac $A$-bundle  with Dirac operator $D$ and $E\to M\setminus K$ a $B$-bundle of bounded variation defined outside of a compact subset $K\subseteq M$. Then the index of the twisted operator $D_E$ is
\[\ind_{/\K}(D_E)=\llbracket D;B\rrbracket_{/\K/C_0}\circ \llbracket E\rrbracket_{/C_0}\,.\]
\end{thm}

\begin{proof}
We can find $W,\tilde M,\tilde S$ as in the construction above but with the additional property that $P|_W$ extends to a projection $P\in\cA^\infty(\tilde M,B)$.
By Theorem \ref{mainthm:indexformula}, $\ind(\tilde D_E)=\llbracket \tilde D;B\rrbracket\circ[P]$ and by Lemma \ref{lem:comparemodKandnotmodK} and Corollary \ref{cor:indexprojection} we conclude
\begin{align*}
\ind_{/\K}(D_E)&=\pi_{\tilde M,W,A\hatotimes B}\circ \llbracket \tilde D;B\rrbracket_B\circ[P]
=\llbracket D;B\rrbracket_{/\K/C_0}\circ\tau_{\tilde M,W,B}\circ [P]
\\&=\llbracket D;B\rrbracket_{/\K/C_0}\circ \llbracket E\rrbracket_{/C_0}\,.
\qedhere
\end{align*}
\end{proof}

\section{Vector bundles of vanishing variation}\label{sec:duality}
In this section we specialize even further to operators twisted by $B$-bundles of vanishing variation. To make these notions precise, we recall some terminology from \cite[Chapter 5]{RoeCCITCRM} and \cite[Section 3]{EmeMey}:
\begin{defn}
Let $X$ be a proper metric space.
\begin{itemize}
\item Let $f\colon X\to Y$ be a map into another metric space $Y$ and $R>0$. 
The \emph{$R$-variation} of $f$ is the function
\[\operatorname{Var}_R\colon X\to [0,\infty)\,,\quad x\mapsto\sup\{d_Y(f(x),f(y))\mid d_X(x,y)\leq R\}\,.\]
The map $f$ is said to have \emph{vanishing variation} if $\operatorname{Var}_R$ vanishes at infinity for all $R>0$.
\item Let $X$ be a proper metric space and $B$ a $\Z_2$-graded unital $C^*$-algebra. The \emph{stable Higson compactification} of $X$ with coefficients in $B$ is the $C^*$-algebra
\[\cbar(X,B):=\{f\in C_b(X,B\hatotimes\widehat\K)\mid f\text{ has vanishing variation}\}\]
and the \emph{stable Higson corona} is
\[\cfrak(X,B):=\cbar(X,B)/C_0(X,B\hatotimes\widehat\K)\,.\]
\end{itemize}
\end{defn}
In contrast to the original definition, we use the graded $C^*$-algebra $\widehat\K$  instead of the ungraded $\K$ here. This difference is irrelevant from an $E$-theoretic point of view because of stability.

Recall from \cite{WulffCoassemblyRinghomo} that there are multiplication $*$-homomorphisms
\[\nabla_{B,C}\colon \cfrak(X,B)\hatotimes\cfrak(X,C)\to\cfrak(X,B\hatotimes C)\,,\quad \overline{f}\hatotimes\overline{g}\mapsto\overline{x\mapsto f(x)\hatotimes g(x)}\]
which are unique up to homotopy. They induce multiplications
\begin{equation}\label{eq:ringstructure}
K_*(\cfrak(X,B))\otimes K_*(\cfrak(X,C))\to K_*(\cfrak(X,B\hatotimes C))
\end{equation}
in $K$-theory and in particular for $B=C=\C$ (or $B$ and $C$ Clifford algebras) we see that $K_*(\cfrak(X,\C))$ is a $\Z_2$-graded, graded commutative ring.

In the context of index theory, a more analytic description of these notions is necessary (cf.\ remarks at the beginning of \cite[Chapter 5]{RoeCCITCRM}). This is available in the context of complete Riemannian manifolds of bounded geometry.

\begin{lem}\label{lem:BG}
Assume $M$ is a complete Riemannian manifold of bounded geometry. Then 
\[\cbar^{\infty}(M,B):=\{f\in C^\infty_b(M,B\hatotimes\widehat\K)\mid \grad(f)\text{ vanishes at }\infty\}\]
is a dense subalgebra of $\cbar(X,B)$.
\end{lem}
\begin{cor}
If $M$ has bounded geometry, then
$\cfrak(M,B)$ is a sub-$C^*$-algebra of $\cA_{/C_0}(M,B)$. We denote its inclusion by $\iota_B$.
\end{cor}
\begin{proof}[Proof of Lemma \ref{lem:BG}]
Clearly $\cbar^{\infty}(M,B)\subseteq\cbar(M,B)$. 

Now let $f\in\cbar(M,B)$. 
Let $\mathcal{Y}_r$ and $K_{r,R}$ ($r,R>0$) be as in Lemma \ref{lem:BGmainproperties} and choose $r>0,Y\in\mathcal{Y}_r,\psi_y,C$ as in Lemma \ref{lem:BGPartitionofUnity}.
We define a new function 
$\hat f:=\sum_{y\in Y}f(y)\psi_y$.
It is clearly smooth and because of $\sum_{y\in Y}\grad(\psi_y)=\grad\left(\sum_{y\in Y}\psi_y\right)=0$ we can calculate
\[\grad(\hat f)(x)=\sum_{y\in Y}\grad(\psi_y)(x)\hatotimes f(y)=\sum_{y\in Y}\grad(\psi_y)\hatotimes (f(y)-f(x))\]
and obtain the pointwise norm estimate
\[\|\grad(\hat f)(x)\|\leq K_{r,R}\cdot C\cdot\operatorname{Var}_r(f)(x)\]
and this vanishes at infinity. Thus, $\hat f\in \cbar^{\infty}(M,B)$. Similarly, 
\[f-\hat f=\sum_{y\in Y}(f(y)-f)\psi_y\]
is pointwise bounded in norm by $K_{r,R}\cdot\operatorname{Var}_r(f)$ and therefore also vanishes at infinity. The claim now follows from the fact that $f-\hat f$ can be approximated in norm by smooth functions of compact support.
\end{proof}

We can now define the bundles corresponding to elements of the $K$-theory of the stable Higson corona:
\begin{defn}
Let $P\in \cbar^\infty(M,B)$ be projection valued outside of a compact subset $K\subseteq M$.
The $B$-bundle $E\to M\setminus K$ whose fibre at $x$ is the image of $P_x$ in $B\hatotimes\widehat\ell^2$ is called a \emph{$B$-bundle of vanishing variation defined outside of a compact subset}.
Its $K$-theory class $\llbracket E\rrbracket_{\cfrak}\in K(\cfrak(M,B))$ is defined to be the class of the projection $\overline{P}\in\cfrak(M,B)$ determined by $P$.
\end{defn}
This notion is obviously  a special case of $B$-bundles of bounded variation and the inclusion $\iota_B$ of the previous corollary maps $\llbracket E\rrbracket_{\cfrak}$ to $\llbracket E\rrbracket_{/C_0}$.

We now turn to the construction of a module multiplication which allows us to state the main result of this section.
\begin{lem}
Let $M$ be a complete Riemannian  manifold of bounded geometry. 
We let $C^*(M,A)$ and $\cbar(M,B)$ act on $\cH:=L^2(M,(A\hatotimes\widehat\ell^2)\hatotimes (B\hatotimes\widehat\ell^2))$ in the obvious way.
Then 
\begin{align*}
\mathfrak{m}_{A,B}\colon C^*_{/\K}(M,A)\hatotimes\cfrak(M,B)&\to C^*_{/\K}(M,A\hatotimes B)
\\\overline{T}\hatotimes \overline{f}&\mapsto\overline{Tf}
\end{align*}
defines a $*$-homomorphism. Furthermore, the following associativity diagram commutes up to homotopy:
\[\xymatrix@C=3.5em{
C^*_{/\K}(M,A)\hatotimes\cfrak(M,B)\hatotimes\cfrak(M,C) \ar[r]^-{\id\hatotimes\nabla_{B,C}} \ar[d]_{\mathfrak{m}_{A,B}\hatotimes\id}
&C^*_{/\K}(M,A)\hatotimes\cfrak(M,B\hatotimes C)\ar[d]^{\mathfrak{m}_{A,B\hatotimes C}}
\\C^*_{/\K}(M,A\hatotimes B)\hatotimes\cfrak(M,C) \ar[r]_-{\mathfrak{m}_{A\hatotimes B,C}}
&C^*_{/\K}(M,A\hatotimes B\hatotimes C)
}\]
\end{lem}
This lemma actually  needs only the metric space version of bounded geometry. Furthermore, we note that if one prefers working with the (unstable) Higson corona (without coefficients) instead of the stable Higson corona with coefficients, then bounded geometry is in fact unnecessary. This follows by modifying our proof using the idea of the proof of \cite[Proposition 5.18]{RoeCCITCRM}. However, this idea is not compatible with stabilization and the presence of  coefficient $C^*$-algebras.
\begin{proof}
We have to prove that the graded commutators $[T,f]$ are compact. Assume that $T$ and $f$ are of homogeneous degree and that $T$ has propagation bounded by $R>0$. 
Choose $Y\in\mathcal{Y}_r$ and let $M=\bigcupdot_{y\in Y} Z_y$ be a decomposition of $M$ into Borel subsets $Z_y\subseteq B_r(y)$.
Denote the characteristic function of $Z_y$ by $\chi_y$
and define
$\check f:=\sum_{y\in Y}f(y)\chi_y$. As before, $f-\check f$ vanishes at infinity and therefore $T(f-\check f)$ and $(f-\check f)T$ are compact. It is thus sufficient to prove that the graded commutators $[T,\check f]$ are compact.
We can write these graded commutators as strongly converging sums of compact operators as follows:
\begin{align}
[T,\check f]&=T\sum_{z\in Y}f(z)\chi_z-(-1)^{\partial T\partial f}\sum_{y\in Y}f(y)\chi_yT\nonumber
\\&=\sum_{y,z\in Y}\chi_yT\chi_z(f(z)-f(y))\nonumber
\\&=\sum_{\substack{y,z\in Y\\d(y,z)\leq R+2r}}\underbrace{\chi_yT\chi_z(f(z)-f(y))}_{\in\K(\cH)}\,.\label{eq:comutatorcompact}
\end{align}
Here we have used that $\chi_yT\chi_z$ vanishes if $d(y,z)>R+2r$.

Let $\varepsilon>0$ and choose $L\subseteq Y$ finite and with the property that $\operatorname{Var}_{R+2r}(y)<\varepsilon$ for all $y\in Y\setminus L$. For  $v\in\cH$ and $v_y:=\chi_zv$ we calculate
\begin{align*}
\left\|\sum_{\substack{y,z\in Y\setminus L\\d(y,z)\leq R+2r}}
\right.&\left.\vphantom{\sum_{\substack{y,z\in Y\setminus L\\d(y,z)\leq R+2r}}}
\chi_yT\chi_z(f(z)-f(y))v\right\|^2=
\\
&=\sum_{y\in Y\setminus L}\left\|\sum_{\substack{z\in Y\setminus L\\d(y,z)\leq R+2r}}\chi_yT(f(z)-f(y))v_z\right\|^2
\\&\leq\sum_{y\in Y\setminus L}\left(\sum_{\substack{z\in Y\setminus L\\d(y,z)\leq R+2r}}\|T\|\cdot\varepsilon\cdot\|v_z\|\right)^2
\\&\leq\sum_{y\in Y\setminus L}K_{r,R+2r}\cdot\sum_{\substack{z\in Y\setminus L\\d(y,z)\leq R+2r}}\|T\|^2\cdot\varepsilon^2\cdot\|v_z\|^2
\\&\leq K_{r,R+2r}^2\cdot\|T\|^2\cdot\varepsilon^2\cdot \sum_{z\in Y\setminus L}\|v_z\|^2
\\&\leq K_{r,R+2r}^2\cdot\|T\|^2\cdot\varepsilon^2\cdot\|v\|^2\,.
\end{align*}
The second inequality was the inequality between arithmetic and quadratic mean, as the sum over the $z$ within the parenthesis had at most $K_{r,R+2r}$ summands.
For $\varepsilon\to0$ this norm inequality implies that
the sum \eqref{eq:comutatorcompact} converges even in  norm and is therefore compact.

It remains to show the commutativity up to homotopy of the diagram. Note that both compositions map 
\[\overline{T}\hatotimes\overline{f}\hatotimes\overline{g}\in C^*_{/\K}(M,A)\hatotimes\cfrak(M,B)\hatotimes\cfrak(M,C)\]
to the element of $C^*_{/\K}(M,A\hatotimes B\hatotimes C)$ represented by the operator $Tfg$ acting on 
\[\cH':=L^2(M,(A\hatotimes\widehat\ell^2)\hatotimes (B\hatotimes\widehat\ell^2)\hatotimes (C\hatotimes \widehat\ell^2))\]
in the obvious way. However, they are not equal, as $\cH'$ is identified with $L^2(M,A\hatotimes B\hatotimes C\hatotimes\widehat\ell^2)$ in two different ways, via two different isometric isomorphisms $\widehat\ell^2\hatotimes\widehat\ell^2\hatotimes \widehat\ell^2\cong\widehat\ell^2$. Therefore, the diagram commutes only up to homotopy.
\end{proof}

\begin{cor}
There are multiplications
\[(\mathfrak{m}_{A,B})_*\colon K_*(C^*_{/\K}(M,A))\hatotimes K_*(\cfrak(M,B))\to K_*(C^*_{/\K}(M,A\hatotimes B))\]
which are associative with respect to the multiplication
of \eqref{eq:ringstructure}.
In particular, $K_*(C^*_{/\K}(M,A))$ is a (right) module over the ring $K_*(\cfrak(M,\C))$. We denote all these multiplications by the infix notation $x\cdot y:=(\mathfrak{m}_{A,B})_*(x\hatotimes y)$.
\end{cor}
The fact that the module structure pops up  as a right module comes from the fact that we always twist Dirac operators with bundles from the right.

We can now formulate and prove the coarse geometric version of ``calculating indices by module multiplication''. For a similar result in foliation index theory, see \cite{WulffFoliations}.

\begin{thm}\label{mainthm:moduleindexformula}
Let $M$ be a complete Riemannian manifold of bounded geometry, $D$ the Dirac operator of a Dirac $A$-bundle $S\to M\setminus K$ and $E\to M\setminus K$ a
$B$-bundle of vanishing variation defined outside of the compact subset $K\subseteq M$.
Then 
\[\ind_{/\K}(D_E)=\ind_{/\K}(D)\cdot \llbracket E\rrbracket_{\cfrak}\,.\]
\end{thm}
\begin{proof}
We shall prove the formula
\begin{equation}\label{eq:proovingmoduleformula}
\llbracket D;B\rrbracket_{/\K/C_0}\circ \iota_{B}=\mathfrak{m}_{A,B}\circ(\ind_{/\K}(D)\hatotimes \id_{\cfrak(M,B)})\,,
\end{equation}
from which the claim follows by composing with $\llbracket E\rrbracket_{\cfrak}$ and applying Theorem \ref{mainthm:modKindexformula} and the definition of the module structure.

Let $W,\tilde M,\tilde S,\tilde D$ be as in the construction of the $E$-theory class of $D$. Then the index $\ind_{/\K}(D)$ is represented by the asymptotic morphism
\[\cS\to \fA(C^*_{/\K}(M,A))\,,\quad\phi\mapsto \overline{t\mapsto\overline{V_{\tilde M,W,A}^*\phi(t^{-1}\tilde D)V_{\tilde M,W,A}}}\]
and so the right hand side of \eqref{eq:proovingmoduleformula} is represented by
\[\cS\hatotimes\cfrak(M,B)\to \fA(C^*_{/\K}(M,A\hatotimes B))\,,\quad\phi\hatotimes\overline{f} \mapsto \overline{t\mapsto\overline{V_{\tilde M,W,A}^*\phi(t^{-1}\tilde D)V_{\tilde M,W,A}f}}\,.\]
The left hand side, on the other hand, is represented by the asymptotic morphism
\begin{align*}
\cS\hatotimes\cfrak(M,B)&\to \fA(C^*_{/\K}(M,A\hatotimes B)\,,
\\\phi\hatotimes\overline{f}&\mapsto \overline{t\mapsto\overline{ V_{\tilde M,W,A\hatotimes B}^*\phi(t^{-1}\tilde D)f\omega V_{\tilde M,W,A\hatotimes B}}}
\end{align*}
with $\omega$ as in Lemma \ref{lem:comparemodKandnotmodK}. The claim follows just like in the proof of Lemma \ref{lem:comparemodKandnotmodK}, because $V_{\tilde M,W,A}^*\phi(t^{-1}\tilde D)V_{\tilde M,W,A}f$ and $V_{\tilde M,W,A\hatotimes B}^*\phi(t^{-1}\tilde D)f\omega V_{\tilde M,W,A\hatotimes B}$ differ only by a compact operator.
\end{proof}

Emerson and Meyer invented the stable Higson corona as a dual to the Roe algebra. This duality emerges in the shape of a pairing
\[\tilde K_{-*+1}(\cfrak(X,\C))\times K_*(C^*(X,\C))\to\Z\]
which is compatible with assembly and co-assembly \cite[Theorem 6.1]{EmeMey}. In our context, i.\,e.\ after unreducing these groups using the $X^\to$-construction and focussing on the manifold case, the pairing can be described in terms of the module multiplication:
\begin{defn}\label{def:Pairings}
For a connected, complete Riemannian manifold $M$ of bounded geometry we define the pairing
\[K_*(C_{/\K}^*(M,\C))\hatotimes K_{-*+1}(\cfrak(M,\C))\to \Z\]
as the composition of the module multiplication with the connecting homomorphism
$K_{1}(C_{/\K}^*(M,\C))\to K_0(\widehat\K)\cong\Z$.
In more generality, the pairing can be defined as 
\[K_i(C_{/\K}^*(M,A))\hatotimes K_j(\cfrak(M,B))\to K_{i+j-1}(\K(M,A\hatotimes B))\,.\]
\end{defn}
These pairings are stabilized versions of the pairing of 
\cite{YuKIndices} and may be used to detect non-vanishing of indices.

Furthermore, for a non-compact, connected complete Riemannian manifold of bounded geometry, the short exact sequence \eqref{eq:modKseq1}, the product formula of Theorem \ref{mainthm:moduleindexformula} and the remark following Definition \ref{def:modKindex} immediately
yield the following interpretation of the index pairing:
\begin{thm}\label{thm:pairinginterpretation}
Let $M$ be a complete non-compact Riemannian manifold of bounded geometry, $D$ the Dirac operator of a Dirac $A$-bundle $S\to M\setminus K$ and $E\to M\setminus K$ a
$B$-bundle of vanishing variation defined outside of the compact subset $K\subseteq M$.
Then the pairing between $\ind_{/\K}(D)$ and $\llbracket E\rrbracket_{\cfrak}$ is an obstruction to the extendibility of $D_E$ to all of $M$.\qed
\end{thm}

\section{Application to non-vanishing of indices}\label{sec:Application}
In this final Section we want to illustrate how the module structure can be used to prove the non-vanishing of indices. The main technical tool we use is the partitioned manifold theorem, see \cite[Theorem 4.4]{RoeITCGTM} for the case without coefficient $C^*$-algebras $A$ and \cite[Theorem 2.6]{Zadeh} for a version with coefficients.  
We need it in the following version, which can be proven by adapting the proof of \cite[Theorem 5.15]{ZeidlerPSCProductSecondary}.
\begin{thm}
Let $M$ be a complete Riemannian manifold and $N\subseteq M$ a connected compact codimension $1$ submanifold such that $M\setminus N$ decomposes into two open subsets $M^\pm\subseteq M$ with common boundary $\partial M^\pm=N$. Assume further that the metric of $M$ restricts to the product metric on a tubular neighborhood $N\times(-\varepsilon,\varepsilon)$ of $N$. Let $D_M$ be a $A\hatotimes\Cl_1$-linear Dirac operator over $M$ and $D_N$ a $A$-linear Dirac operator over $N$ such that $D_M$ restricts to $D_N\times D_{(-\varepsilon,\varepsilon)}$ over the tubular neighborhood. Then the Mayer--Vietoris connecting homomorphism
\[\partial_{MV}\colon K(C^*(M,A)\hatotimes \Cl_1)\to K(A)\]
maps $\ind(D_M)$ to $\ind(D_N)$.
\end{thm}
\begin{rem}
It might be worthwhile to also construct a lift of  this theorem from indices to fundamental $E$-theory classes, just like we lifted bordism invariance in Theorem \ref{thm:BordismInvariance}.
\end{rem}

Now let $M$ be a complete Riemannian manifold and $D$ a $A\hatotimes\Cl_1$-linear Dirac operator over the complement $M\setminus K$ of a compact subset. Choose $W,\tilde M,\tilde D$ as in the construction of $\ind_{/\K}(D_M)$.
We want to apply the partitioned manifold theorem to $\tilde M$ partitioned by $\partial W$, so we assume that $\partial W$ is connected. Metric and operator might have to be perturbed over a compact subset in order to be of product type over the tubular neighborhood, but this does not effect the construction of the index of $D_M$.

Denote by $C^*(W\subseteq\tilde M,A)$, $C^*(\tilde M\setminus W\subseteq\tilde M,A)$, $C^*(\partial W\subseteq\tilde M,A)$ the $C^*$-subalgebras of $C^*(\tilde M,A)$ generated by the operators whose support are contained in some $R$-neighborhood of $W$, $\tilde M\setminus W$, $\partial W$, respectively. They are ideals in  $C^*(\tilde M,A)$ and $C^*(\partial W\subseteq \tilde M,A)\cong A\hatotimes \widehat\K$, because $\partial W$ is compact and connected. 
One way of obtaining the Mayer--Vietoris boundary map is by composing functoriality under
\[C^*(\tilde M,A)\to \frac{C^*(\tilde M,A)}{C^*(\tilde M\setminus W\subseteq\tilde M,A)}
\cong\frac{C^*(W\subseteq\tilde M,A)}{C^*(\partial W\subseteq\tilde M,A)}\]
with the connecting homomorphism associated to the short exact sequence 
\[0\to C^*(\partial W\subseteq\tilde M,A) \to C^*(W\subseteq\tilde M,A)\to \frac{C^*(W\subseteq\tilde M,A)}{C^*(\partial W\subseteq\tilde M,A)}\to 0\,.\]

Now note that there is a canonical isomorphism 
\[\frac{C^*(W\subseteq\tilde M,A)}{C^*(\partial W\subseteq\tilde M,A)}\cong C_{/\K}^*(M,A)\]
and using it we see that $\partial_{MV}$ factors as the composition of $\pi_{\tilde M,W,A\hatotimes\Cl_1}=\pi_{\tilde M,W,A}\hatotimes\id_{\Cl_1}$ and the connecting homomorphism $\partial\colon K(C_{/\K}^*(M,A\hatotimes\Cl_1))\to K(A)$ which we already encountered in Definition \ref{def:Pairings}. Finally, Corollary \ref{cor:indexprojection} yields the following consequence of the partitioned manifold theorem:
\begin{lem}\label{lem:HalfofPartitionedManifold}
$\partial(\ind_{/\K}(D_M))=\ind(D_N)$.\qed
\end{lem}
This lemma is an important tool for proving non-vanishing of the index.

We conclude with the following concrete example, which shows how the lemma together with the module structure can be used to prove non-vanishing of indices:

\begin{exam}\label{ex:Hopfbundletwist}
Let $\slashed D$ be the unique $\Cl_3$-linear spinor Dirac operator over $\R^3$. We have $\partial(\ind_{/\K}(\slashed D))=0$, because $\slashed D$ is defined on all of $M$. Thus, $\partial$ alone is not sufficient for detecting the non-vanishing of the index.

Now let $H\to S^2$ be the two-dimensional vector bundle defining the nontrivial class in $K(S^2)$. Denote by $E\to\R^3\setminus\{0\}$ the pullback of $H$ under the radial projection. It satisfies the prerequisites of Theorem \ref{mainthm:moduleindexformula} and therefore
$\ind_{/\K}(\slashed D_E)=\ind_{/\K}(\slashed D)\cdot\llbracket E\rrbracket_{\cfrak}$ by module multiplication.

To perform the constructions above, we may take $W\subseteq \R^3$ to be the complement of the unit ball and modify the metric in a tubular neighborhood of $\partial W=S^2$ such that it is of product type. 
On this tubular neighborhood $S^2\times(-\varepsilon,\varepsilon)$, $\slashed D$ is equal to $\slashed D'\times D_{(-\varepsilon,\varepsilon)}$, where $\slashed D'$ is the Dirac operator associated to the spin structure bounding the disk.
Consequently, $\slashed D_E$ restricts to $\slashed D'_H\times D_{(-\varepsilon,\varepsilon)}$ and Lemma \ref{lem:HalfofPartitionedManifold} yields
$\partial(\ind_{/\K}(\slashed D_E))=\ind(\slashed D'_H)\in\Z$.

It is well-known that the latter is non-zero. Thus, $\ind_{/\K}(\slashed D_E)\not=0$ and by exploiting the module structure also $\ind_{/\K}(\slashed D)\not=0$. An alternative way of expressing this is to say that
$\langle\ind_{/\K}(\slashed D),\llbracket E\rrbracket_{\cfrak}\rangle=\partial(\ind_{/\K}(\slashed D_E))=\ind(\slashed D'_H)\not=0$ $\Rightarrow\ind_{/\K}(\slashed D)\not=0$.
\end{exam}


{
\bigskip

\noindent
Author's affiliation at the time of the research:

\noindent
\textsc{Instituto de Matem\'aticas (Unidad Cuernavaca), 
Universidad Nacional Aut\'onoma de M\'exico, 
Avenida Universidad s/n, Colonia Lomas de Chamilpa, 
62210 Cuernavaca, Morelos, Mexico}

\noindent
\textit{E-mail address:} \textsf{christopher.wulff@im.unam.mx}
}

{
\bigskip

\noindent
Current address:
\\
\textsc{
Mathematisches Institut, 
Georg--August--Universit\"at G\"ottingen,
Bunsenstr. 3-5, 
D-37073 G\"ottingen, 
Germany}

\noindent
\textit{E-mail address:} \textsf{christopher.wulff@mathematik.uni-goettingen.de}
}

\end{document}